\documentclass[a4paper,fleqn]{cas-dc}

\usepackage[authoryear,longnamesfirst]{natbib}
\usepackage{amsmath,amssymb,amsfonts,amsthm}
\usepackage{algpseudocode}
\usepackage{algorithm}
\usepackage{array}
\usepackage[caption=false,font=normalsize,labelfont=sf,textfont=sf]{subfig}
\usepackage{textcomp}
\usepackage{stfloats}
\usepackage{url}
\usepackage{verbatim}
\usepackage{graphicx}
\usepackage{bm}
\usepackage{booktabs}
\usepackage{dsfont}
\usepackage{textcomp}
\usepackage{enumitem}
\usepackage{hyperref}

\usepackage[markup=default,commandnameprefix=always,authormarkup=none]{changes}

\newtheorem{theorem}{Theorem}
\newtheorem{lemma}{Lemma}

\newtheorem{definition}{Definition}

\newtheorem{remark}{Remark}

\def\tsc#1{\csdef{#1}{\textsc{\lowercase{#1}}\xspace}}
\tsc{WGM}
\tsc{QE}

\begin{document}
\let\WriteBookmarks\relax
\def\floatpagepagefraction{1}
\def\textpagefraction{.001}

\shorttitle{}    

\shortauthors{}  

\title [mode = title]{Distributionally Robust Fault Detection Trade-off Design with Prior Fault Information}  



%

\author[1]{Yulin Feng}[orcid=0009-0009-5218-3279]
\ead{fyl23@mails.tsinghua.edu.cn}

\author[1]{Hailang Jin}[orcid=0000-0002-7705-3843]
\ead{hljin@tsinghua.edu.cn}

\author[2]{Steven X. Ding}[orcid=0000-0002-5149-5918]
\ead{steven.ding@uni-due.de}

\author[1]{Hao Ye}[orcid=0000-0003-1607-244X]
\ead{haoye@tsinghua.edu.cn}

\author[1]{Chao Shang}[orcid=0000-0003-3905-4631]
\ead{c-shang@tsinghua.edu.cn}
\cormark[1]




\affiliation[1]{organization={Department of Automation, Beijing National Research Center for Information Science and Technology, Tsinghua University},
            postcode={100084},
            state={Beijing},
            country={China}}

\affiliation[2]{organization={Institute for Automatic Control and Complex Systems, University of Duisburg-Essen},
            postcode={47057},
            state={Duisburg},
            country={Germany}}





\cortext[1]{Corresponding author}



\begin{abstract}
The robustness of fault detection algorithms against uncertainty is crucial in the real-world industrial environment. Recently, a new probabilistic design scheme called distributionally robust fault detection (DRFD) has emerged and received immense interest. Despite its robustness against unknown distributions in practice, current DRFD focuses on the overall detectability of all possible faults rather than the detectability of critical faults that are \textit{a priori} known. Henceforth, a new DRFD trade-off design scheme is put forward in this work by utilizing prior fault information. The key contribution includes a novel distributional robustness metric of detecting a known fault and a new relaxed distributionally robust chance constraint that ensures robust detectability. Then, a new DRFD design problem of fault detection under unknown probability distributions is proposed, and this offers a flexible balance between the robustness of detecting known critical faults and the overall detectability against all possible faults. To address the resulting semi-infinite chance-constrained problem, we first reformulate it to a finite-dimensional problem characterized by bilinear matrix inequalities. Subsequently, a tailored heuristic solution algorithm is developed, which includes a sequential minimization procedure and an initialization strategy. Finally, case studies on a simulated three-tank system and a real-world battery cell are carried out to showcase the effectiveness of the proposed heuristic algorithm and the advantages of our DRFD method.
\end{abstract}





\maketitle

\section{Introduction}
The growing complexity and interconnectivity of modern industrial processes underscore the critical importance of ensuring system safety and reliability. Advanced fault detection (FD) techniques that can effectively handle various types of uncertainty are thus in urgent need. In the past decades, model-based and data-driven FD methods have been extensively investigated and widely applied in a variety of fields. A summary of related methods can be found in \cite{ding2008model,ding2014data}. 
In general, current FD methods are built upon the so-called analytical redundancy, and their main steps include residual generation and residual evaluation \citep{ding2008model}. The aim of residual generation is to produce an informative indicator that captures the unpredictable variations in dynamic systems and ensures sensitivity to faults. Various methods have been developed for the design of residual generators, including observer-based methods \citep{ding2019application}, parity space-based methods \citep{chow1984analytical}, and subspace identification methods (SIM) \citep{ding2014dsubspace,jin2023small}. Residual evaluation is used to determine whether a fault has occurred or not by means of alarm logic. 

In practical dynamical systems, there exist various types of uncertain disturbances. Because of this, false alarm rate (FAR) and fault detection rate (FDR) are important performance indices in a probabilistic context and the current research focuses on probabilistic FD design that enjoys robustness against uncertainty. For a residual generator, computing its FAR and FDR entails full knowledge of the distribution of disturbance. However, this is usually unavailable in engineering practices, and the Gaussian assumption has been mostly used as an approximation \citep{ding2014data,yang2019new,fezai2020online}. On this basis, FD via statistical hypothesis testing has been extensively explored, such as the generalized likelihood ratio test \citep{fezai2020online} and Mann-Whitney test \citep{yang2019new}.
Unfortunately, the Gaussian assumption itself may not hold true in practice, resulting in unexpected performance loss of FD. 

To overcome this limitation, distribution-free methods that consider the worst-case performance against uncertainty in a prescribed compact set
have been profoundly investigated  \citep{wang2020dynamic,hast2015detection}. Nevertheless, lacking prior knowledge of the distribution may result in over-conservatism. 
As a hotspot in operations research community, distributionally robust optimization (DRO) offers a systematic and effective solution to address these drawbacks; more importantly, DRO does not require estimating the full probability distribution but introduces an ambiguity set containing all possible probability distributions with similar statistical attributes instead. The two core elements are the shape and the size of the ambiguity set, by which the expected performance under the worst-case distribution is improved. It is a common practice that the former is chosen by decision-makers, such as moment-based constraints \citep{delage2010distributionally,zymler2013distributionally} and metric-based ball \citep{mohajerin2018data,gao2023distributionally,lam2019recovering}. The advances in DRO have recently inspired the distributionally robust fault detection (DRFD) scheme \citep{shang2021distributionally,wan2021distributionally,shang2022generalized,xue2022integrated}. It concurrently yields a residual generator and the alarm threshold by solving an integrated DRO problem while being robust to uncertainty in probability distributions. Thanks to duality theory, the DRFD problem can be reformulated as a deterministic optimization problem that is easier to solve than the primal problem.

In addition to effectively controlling FAR, the detectability for known critical faults is also essential in practice. These critical faults often include high-stake faults and high-frequency faults. The former are known to cause significant consequences to the system security and can be obtained from expert knowledge. For example,  a power grid is known to be significantly influenced by transformer failure, thus ensuring its timely detection is of great importance \citep{yu2024automatic}. The latter can be observed from historical data and obtained using various data-driven methods \citep{feng2020domain,van2021data,yao2024domain}. However, in the existing DRFD designs \citep{shang2021distributionally,shang2022generalized}, only the overall fault detectability of unknown faults is considered, and the information of critical faults is omitted. Hence, it is worth investigating a robust detection scheme considering detectability against both known and unknown faults in a distributionally robust fashion.
To address these issues, this paper seeks to achieve an appropriate trade-off between detecting known critical faults and detecting other unknown faults within the DRFD framework. The main contributions of this paper are listed as follows:
\begin{itemize}
    \item  We formalize a new design problem that manages the trade-off between detecting critical faults and the overall detectability against unknown faults. As a cornerstone of this scheme, we introduce a novel metric, derived from a relaxed distributionally robust chance constraint (DRCC) that leverages prior fault information, to quantify the distributional robustness for detecting known critical faults.
     \item The design problem is semi-infinite and cannot be solved straightforwardly, thus we derive an exact reformulation of the proposed DRFD design problem, which leads to a non-convex problem due to the presence of bilinear matrix inequalities. To address the resulting problem, a tailored heuristic algorithm is developed, involving a sequential minimization procedure and a dedicated initialization strategy.
    \item The detection performance and efficiency of our method are demonstrated through case studies on a simulated three-tank system and a realistic lithium-ion battery cell. These results highlight that the proposed DRFD provides a systematic approach to account for the sensitivity against known critical faults and unknown ones.
\end{itemize}

The rest of this article is organized as follows. Section \ref{preli} revisits the basics of residual generation and evaluation for dynamic systems, as well as the design formulation of DRFD. The main results of this work are shown in Section \ref{mainr}, and case studies are presented in Sections \ref{simu1} and \ref{simu2}. Section \ref{concl} concludes this article.

\textit{Notations:} We use $\mathbb{N}_{i:j}$ to denote the set of non-negative integer indexes $\{i,\cdots,j\}$. $P \succeq 0$  denotes the positive semi-definiteness of a symmetric $P$, $P(:,i:j)$ denotes the submatrix of $P$ from the $i$th column to the $j$th column, and $\lVert P \rVert_{F}$ denotes the Frobenius norm.  ${\rm tr}(\cdot)$  and ${\rm det}(\cdot)$ are the trace and the determinant operator of a square matrix. For a vector $f$,  $\lVert f \rVert^2_{P} =f^\top Pf$ and
 $\lVert f \rVert^2$ = $f^\top f$. $\delta_\xi$ represents the Dirac distribution concentrated at $\xi$. For a given closed set $\mathcal{S}$, the indicator function $\mathbb{1}_{{\rm cl}(\xi \in \mathcal{S})}=1$ if $\xi \in \mathcal{S}$; and $0$ otherwise.
 For a discrete-time signal $x(k)$, the concatenated vector is denoted by $x_s(k) = \left[x(k-s)^\top \cdots x(k)^\top \right]^\top \in \mathbb{R}^{(s+1)n_x }$ and the block Hankel matrix of $ N \in \mathbb{N}_+$ interval  $X_{k,s} = \left[x_s(k-N+1) \cdots x_s(k) \right] \in \mathbb{R}^{(s+1)n_x \times N}$. The lower triangular block-Toeplitz operator $\mathcal{T}(A, B, C, D)$ is defined as:
\begin{equation*}
    \mathcal{T}(A, B, C, D) =\begin{bmatrix}D&0&0&\cdots&0\\CB&D&0&\cdots&0\\CAB&CB&D&\cdots&0\\\vdots&\ddots&\ddots&\ddots&\vdots\\CA^{s-1}B&\cdots&CAB&CB&D\end{bmatrix} .
\end{equation*}

\section{Preliminaries and Problem Formulation} \label{preli}
\subsection{Residual Generation and Evaluation}
Consider a stochastic discrete-time linear time-invariant (LTI) system described by
\begin{equation}\begin{cases} {x}(k+1)=A{x}(k)+B {u}(k)+B_dd(k)+B_{f} {f}(k)\\{y}(k)=C {x}(k)+D {u}(k)+D_dd(k)+D_f{f}(k)&\end{cases} \label{eq_statepace}\end{equation}
where ${x}\in\mathbb{R}^{n_{x}}$, ${y}\in\mathbb{R}^{n_y}$, ${u}\in\mathbb{R}^{n_u}$, $d\in\mathbb{R}^{n_d}$, and ${f}\in\mathbb{R}^{n_f}$ are process state, measured output, control input, stochastic disturbance, and additive faults, respectively. 
$A$, $B$, $B_d$, $B_f$, $C$, $D$, $D_d$, and $D_f$ are time-invariant matrices having appropriate dimensions. The transfer function matrix from $u$ to $y$ is denoted by $G_u\left(z\right)$, where $z$ is the time-shift operator. It is assumed that $(A, C)$ is observable. 

\begin{remark}
The assumption of additive faults in \eqref{eq_statepace} is rather common in fault diagnosis literature \citep{ding2008model,ding2014data}. An underlying fault enters the system through matrices $B_f$ and $D_f$, which encompass a wide range of system malfunctions.  
For instance, actuator faults can be modeled by setting $B_{f}=B$ and $D_f=0$, whereas sensor faults are represented with $B_{f}=0$ and $D_f=I$. 
This unified representation paves the way for developing a detection framework capable of handling multiple fault types.
\end{remark}

The stable kernel representation (SKR) is a general way to represent the residual generator for LTI systems \citep{ding2008model,ding2014data}. Under the  disturbance- and fault-free conditions, the following relation holds:
\begin{equation*}
\mathcal{K}(z)\begin{bmatrix} u(z)\\y(z) \end{bmatrix} \equiv 0,~\forall u,
\end{equation*}
where $\mathcal{K}(z)=\begin{bmatrix}-\hat{N}(z)&\hat{M}(z)\end{bmatrix} $ can be established from the left coprime factorization of the transfer function matrix $G_u\left(z\right) = \hat{M}^{-1}(z) \hat{N}(z)$. Motivated by SKR, a general residual generator is expressed in the following form:
\begin{equation*}
r(z) = \mathcal{K}(z)\begin{bmatrix} u(z)\\y(z) \end{bmatrix} ,
\end{equation*}
where $r(z) \in \mathbb{R}^{n_r}$ is the residual, and its design form in the time domain essentially takes the form of an affine function combination of fault and disturbance:
\begin{equation}
    r(k)=P[Wd_s(k)+Vf_s(k)],\label{eq:resi_p}
\end{equation}
where $P$ is the projection matrix that parameterizes the residual generator, and shall be deliberately designed to preserve ensuring sensitivity to faults and robustness to disturbances. The coefficient matrices $\{W, V\}$ are known, and it is normally assumed that $W$ has full row rank. One can construct $\{W, V\}$ from system matrices or available past fault-free data, e.g. using the parity space method \citep{chow1984analytical} or SIM \citep{ding2014dsubspace}.
Given known state-space matrices, the parity relation of a given order $s \geq n_x$ is established as follows:
\begin{equation} \label{eq_parity}
y_s(k) = \Gamma_s x(k-s)+ H_{u,s}u_s(k) + H_{d,s}d_s(k) + H_{f,s}f_s(k),
\end{equation}
where
\begin{align*}
    &\Gamma_s = \begin{bmatrix}C^\top & (CA)^\top & \dots & \left(CA^{s}\right)^\top\end{bmatrix}^\top,\\
    &H_{u,s}= \mathcal{T}(A, B, C, D),\\
    &H_{d,s}= \mathcal{T}(A, B_d, C, D_d),\\
    &H_{f,s}= \mathcal{T}(A, B_f, C, D_f).
\end{align*}
Thus, the residual generator can be designed based on the parity relation \eqref{eq_parity}:
\begin{equation*} 
r(k) = P \cdot \Gamma^{\perp}_s \left[  y_s(k) - H_{u,s}u_s(k) \right],
\end{equation*}
where $\Gamma^{\perp}_s$ is the basis matrix of the left null space of $\Gamma$, satisfying ${\rm rank}(\Gamma^{\perp}_s) = sn_y-n_x$ and $\Gamma^{\perp}_s \Gamma_s = 0$. 

For systems with unknown state-space matrices, $\{W, V\}$ can be identified from the subspace composed of available historical input-output data
\begin{align*}
    Z_{\rm p} = \begin{bmatrix}
        U_{\rm p}\\
        Y_{\rm p}
    \end{bmatrix}=\begin{bmatrix}
        U_{k-s_{\rm p}-1,s_{\rm p}}\\
        Y_{k-s_{\rm p}-1,s_{\rm p}}
    \end{bmatrix}
    ,~Z_{\rm f} = \begin{bmatrix}
        U_{\rm f}\\
        Y_{\rm f}
    \end{bmatrix}=\begin{bmatrix}
        U_{k,s_{\rm f}}\\
        Y_{k,s_{\rm f}}
    \end{bmatrix}.
\end{align*}
For example, the data-driven SKR can be realized by QR factorization of past and future datasets $\{Z_{\rm p}, Z_{\rm f}\} $ \citep{ding2014data,ding2014dsubspace} :
\begin{equation*}\begin{bmatrix}Z_{\rm p}\\U_{f}\\Y_{f}\end{bmatrix}=\begin{bmatrix}R_{11}&0&0\\R_{21}&R_{22}&0\\R_{31}&R_{32}&R_{33}\end{bmatrix}\begin{bmatrix}Q_1\\Q_2\\Q_3\end{bmatrix}.\end{equation*}
By solving the following equation for $\mathcal{K}_{{\rm d},s}  $:
\begin{equation}
\mathcal{K}_{{\rm d},s}  \begin{bmatrix}R_{21}&R_{22}\\R_{31}&R_{32}\end{bmatrix} = 0.
\end{equation}
the residual generator can be constructed as
\begin{align*}
    r(k)&=P \cdot \mathcal{K}_{{\rm d},s}  \begin{bmatrix}u_s(k)\\y_s(k)
\end{bmatrix}.
\end{align*}
Additionally, one should set $W$ as an identity matrix by taking the residuals $R_{33}Q_3$ as samples of $d_s(k)$. For sensor bias and actuator failure, $V$ can also be readily specified \citep{ding2014dsubspace,shang2022generalized}. After obtaining the residual, the residual evaluation process follows the alarm logic:
\begin{equation*}
    \begin{cases}
    J(r) > J_{{\rm th}} \Rightarrow & \text{fault alarm} \\
    J(r) \leq J_{{\rm th}} \Rightarrow & \text{no alarm}
    \end{cases}
\end{equation*}
where $J_{\rm th}$ is the alarm threshold, and the Euclidean norm evaluation function $J(r) = \lVert r\rVert^2$ is mostly used. Inevitably, the uncertainty in $d(k)$ may lead to unnecessary alarms during fault-free operation and missed alarms when a fault is present. Thus, FAR and FDR are defined as important indicators to evaluate the FD performance in handling this inherent conflict.

\begin{definition}[FAR and FDR, \cite{ding2008model}]
Given the threshold $J_{\rm th}$, the FAR and FDR for a particular fault ${f} \neq 0$ of an FD design is defined as
\begin{align*}    
{\rm FAR} &= \mathbb{P} \{\lVert r \rVert^2 > J_{{\rm th}} \mid f=0\},\\
{\rm FDR} &=  \mathbb{P} \{\lVert r \rVert^2
> J_{{\rm th}} \mid f \neq 0 \}.
\end{align*}
\end{definition}

\subsection{Problem Formulation of DRFD}
A critical point of FD design \eqref{eq:resi_p} is the choice of $P$, around which numerous solutions have been developed. Among previous works, the unified solution is a conventional approach, which maximizes the ratio of the gains of $f$ relative to $d$ \citep{wunnenberg1987sensor}:
\begin{equation}\label{eq_US}
    P^{*}=\arg\min_{P} \frac{\lVert PW \rVert}{\lVert PV \rVert} = \Lambda_{W}^{-1} U_{W}^\top.
\end{equation}
This can be solved by performing singular value decomposition (SVD) of $W=U_{W} \Lambda_{W}V_{W}^\top$  \citep{ding2008model}. After designing $P^*$, the Gaussian assumption with $d \sim \mathcal{N}(0,\Sigma_d)$ is often adopted, and subsequently, $\lVert r \rVert^2$ under fault-free conditions is approximately governed by $\chi^2$-distribution. Consequently, given a prescribed FAR, the alarm threshold $J_{\rm th}$ can be calibrated to a specific upper quantile of a $\chi^2$-distribution \citep{ding2014data,fezai2020online,yang2019new}. 

However, the accuracy of the unified solution \eqref{eq_US} may not be satisfying in engineering practice. The reason is that \eqref{eq_US} is defined using norm-based sensitivity and suffers from over-conservatism. Furthermore, the Gaussian assumption is fragile, leading to performance misspecification when the uncertainties deviate from the Gaussian distribution. To address this issue, DRFD has emerged as an effective scheme \citep{shang2021distributionally}. It regards $\xi = Wd_s(k)$ as the source of uncertainty in residual generation, and thus the relation in \eqref{eq:resi_p} simplifies to:
\begin{equation}
        r=P(\xi+Vf),\label{eq:resi_s}
\end{equation}
where the dependence on the time index $k$ and the order $s$ is dropped for brevity. In engineering practice, the true distribution $\mathbb{P}^*_\xi$ of the disturbance is not known exactly, and thus using an ambiguity set $\mathcal{D}$ to capture the uncertainty of  $\mathbb{P}^*_\xi$ is a common practice in DRO \citep{shang2022generalized,tan2024design}. On this basis, the DRFD design scheme aims at maximizing the overall detectability metric under the worst-case FAR constraint within $\mathcal{D}$:
\begin{equation} \label{eq_DRFD}
\begin{aligned}
\max_{P} &~ \rho\left(P\right)\\
{\rm s.t.} &~ \sup_{\mathbb{P}_{\xi} \in \mathcal{D}}\mathbb{P}_{\xi} \{\lVert r\rVert^2 > 1 \mid f=0\} \leq \alpha
\end{aligned}\tag{DRFD}
\end{equation}
where $\alpha \in (0,1)$ is an upper bound of FAR that is intuitive to specify for the user.
In \eqref{eq_DRFD}, setting $J_{\rm th}=1$ causes no loss of generality. This is because any design with arbitrary matrix $P$ and threshold $J_{\rm th}>0$ can be normalized into an equivalent solution pair $P'={P}/{\sqrt{J_{\rm th}}}$ and $J_{\rm th}'=1$. The DRCC
The objective is to maximize an overall fault detectability metric function $\rho(\cdot)$. It geometricly evaluates the impact of an unknown profile $f$ on the residual, and can be chosen as the Frobenius norm metric \citep{ding2019application}
\begin{equation}
\rho_1(P)=\|PV\|_F^2=\mathrm{tr}(V^\top P^\top PV), \label{eq_rho1}
\end{equation}
  or the pseudo-determinant metric \citep{shang2021distributionally}
 \begin{equation}
 \begin{aligned}
            \rho_{2}(P)& =\log{\rm pdet}(V^\top P^\top PV) \\&=\log\det(\Lambda^{\top}U_{1}^{\top}P^{\top}PU_{1}\Lambda),
 \end{aligned} \label{eq_rho2}
 \end{equation}
where $V = U_1 \Lambda U_2^{\top}$ is the compact SVD and $\Lambda$ is diagonal and non-singular. 

\begin{remark}\label{remark_rho}
The transfer matrix $PV$ essentially maps the fault signal $f$ to the residual $r$. Mathematically, $\rho_1(\cdot)$ is the sum of squares of all singular values of $PV$, and $\rho_2(\cdot)$ amounts to the product of all non-zero singular values. Therefore, maximizing $\rho_2(\cdot)$ avoids excessively small non-zero singular values, and thus encourages a more balanced and uniform sensitivity to different fault directions than maximizing $\rho_1(\cdot)$. On the other hand, maximizing $\rho_1(\cdot)$ is computationally advantageous because it is linear in $\bar P=P^\top P \succeq 0$ \citep{shang2021distributionally}.
\end{remark}

By solving \eqref{eq_DRFD}, the optimal solution $P^*$ can safely keep the FAR below $\alpha$ under routine fault-free conditions for any $\mathbb{P}_\xi$ in the ambiguity set $\mathcal{D}$, while maximizing the impact of unknown fault $f$ on the residual, gives rise to an integrated DRFD design scheme. Central to the DRFD design is the choice of $\mathcal{D}$, and the Wasserstein ambiguity set has been prevalent recently, which critically relies on the Wasserstein distance for evaluating the discrepancy between two distributions.
\begin{definition}[Wasserstein distance, \cite{kantorovich1958space}]
Given two distributions $\mathbb{P}$ and $\mathbb{P}'$, whose random $\xi$ is over the support set $\Xi \subseteq \mathbb{R}^{n_\xi}$, the Wasserstein distance is defined as:
\begin{equation*}
d_{\rm W}(\mathbb{P},{\mathbb{P}}') = \inf_{\mathbb{Q}\in \mathcal{Q}} \left\{ \int_{\Xi^2} \rVert \xi - \xi'\lVert  \mathbb{Q}\left({\rm d}\xi,{\rm d}\xi'\right) \right\},
\end{equation*}
where $\mathbb{Q} \in \mathcal{Q}$ is composed of all the joint distributions of $\xi$ and $\xi'$ with marginal distributions $\mathbb{P}$ and $\mathbb{P}'$.
\end{definition}

The Wasserstein distance is known to be the minimum transport cost of moving the probability mass from $\mathbb{P}$ to $\mathbb{P}'$ in optimal transport theory \citep{mohajerin2018data}. Based on this, the Wasserstein ambiguity set can be defined as follows.

\begin{definition}[Wasserstein ambiguity set, \cite{mohajerin2018data}] \label{def_ambiguity}
Given $N$ independent samples $\{\hat{\xi}_i\}_{i=1}^N$, and Wasserstein radius $\theta$, the Wasserstein ambiguity set is defined as:
\begin{equation*}
    \mathcal{D}_{\rm W}(\theta ; N) = \left\{\mathbb{P}\in \mathcal{M}(\Xi)\left|d_{\rm W}(\mathbb{P},\hat{\mathbb{P}}_{N})\leq\theta\right.\right\},
\end{equation*}
where $\mathcal{M}(\Xi)$ is the probability space supported on $\Xi$, $\theta$ is the radius of the ambiguity set, and $\hat{\mathbb{P}}_{N} = \frac{1}{N} \sum_{i=1}^{N} \delta_{\hat{\xi}_i}$ is the empirical distribution.
\end{definition}

The Wasserstein ambiguity set encapsulates a family of distributions that ``resemble" the empirical distribution $\hat{\mathbb{P}}_{N}$. With the radius $\theta$ appropriately chosen, $\mathcal{D}_{\rm W}(\theta ; N)$ is able to capture the true distribution $\mathbb{P}^*_\xi$ with a high probability in a data-driven fashion.
Indeed,  the uncertainty $\xi$ in \eqref{eq:resi_s} can be seen as the primary residual signal under routine fault-free conditions, so the samples $\{\hat{\xi}_i\}_{i=1}^N$  can be easily collected from an off-the-shelf residual generator to construct $\mathcal{D}_{\rm W}(\theta; N)$ \citep{shang2021distributionally}.

\section{Main result} \label{mainr}
\subsection{A New Trade-off Formulation of DRFD}
In engineering practice, some critical faults may be \textit{a priori} known in advance, and their detectability is of direct interest. However, the design scheme in \eqref{eq_DRFD} only maximizes the overall detectability against \textit{unknown faults} but makes no use of \textit{known} critical fault information. As a result, it may not perform satisfactorily in the presence of such critical faults.

In this work, it is assumed that the profile of a critical fault, e.g. a sensor fault or an actuator fault, is \textit{a priori} known as $\bar{f}$. The goal is to effectively utilize the directional information of $\bar{f}$ to improve the detectability of DRFD design. 
First, a relaxed DRCC is proposed to robustly secure a high FDR of a critical fault $\bar{f}$: 
\begin{equation}
\begin{split}
   \mathbb{P}_\xi \left\{\lVert P\left(\xi + V \bar{f} \right)\rVert^2 > 1\right\}
\geq  \beta - {\frac{1}{\eta}} d_{\rm W}(\mathbb{P}_{\xi},\hat{\mathbb{P}}_{N}),& \\
\forall \mathbb{P}_{\xi} \in \mathcal{M}(\Xi),& 
\end{split}
\label{eq: 5}
\end{equation}
where $\beta \in (0,1)$ denotes a prescribed nominal lower bound of FDR under the empirical distribution $\hat{\mathbb{P}}_N$, and $\eta > 0$ is a decision variable. Intuitively, if the candidate distribution $\mathbb{P}_{\xi} = \hat{\mathbb{P}}_N$, a minimal FDR $\beta$ under $\bar{f}$ can be guaranteed by \eqref{eq: 5}. However, when $\mathbb{P}_{\xi}$ deviates from $\hat{\mathbb{P}}_N$, the FDR under $\bar{f}$ may become lower, and the reduction of FDR is at most $ 1 / \eta \cdot d_{\rm W}(\mathbb{P}_{\xi}, \hat{\mathbb{P}}_{N})$. Furthermore, the larger $\eta$, the less sensitivity of FDR to the deviation of $\mathbb{P}_{\xi}$. In this sense, $\eta$ yields a novel metric evaluating the \textit{distributional robustness} of detecting $\bar{f}$ when $\mathbb{P}^*_{\xi}$ is not exactly known.\footnote{Indeed, $1/\eta$ can be interpreted as the fragility in the robust satisficing scheme \citep{long2023robust}.} To robustly detect the occurrence of a known $\bar{f}$, a large $\eta$ is preferred.

As such, we propose to incorporate \eqref{eq: 5} in the design problem \eqref{eq_DRFD} and maximize $\eta$ jointly with the overall detectability $\rho(\cdot)$. This gives rise to a new DRFD trade-off design, formalized as follows:

\begin{subequations}\label{eq:RS_P}
\begin{align}
\max_{P,\eta>
0} &~ \rho\left(P\right) +  {\gamma}{\eta}  \label{eq:qp0} \\
{\rm s.t.} &~ \sup_{\mathbb{P}_{\xi} \in \mathcal{D}_{\rm W}(\theta, N)}
\mathbb{P}_\xi \left\{\lVert P \xi\rVert^2> 1 \right\} \leq \alpha \label{eq:qp1} \\
\begin{split} 
&~\sup_{\mathbb{P}_{\xi}  \in \mathcal{M}(\Xi)} \mathbb{P}_\xi \left\{\left \lVert  P \left( \xi + V \bar{f} \right)\right \rVert^2 \leq 1\right\}   \\
&~~~\qquad \qquad\quad -  {\frac{1}{\eta}} d_{\rm W}(\mathbb{P}_{\xi},\hat{\mathbb{P}}_{N})
  \leq 1 -\beta  ,\\
\end{split} \label{eq:qp2} 
\end{align}
\end{subequations}
Maximizing the distributional robustness metric $\eta$ in \eqref{eq:RS_P} encourages a low reduction of detectability of the critical fault when the true distribution $\mathbb{P}^*_{\xi}$ deviates from the empirical one $\hat{\mathbb{P}}_N$. In this sense, the objective \eqref{eq:qp0} attempts to trade off robust detection of $\bar{f}$ against the overall detectability $\rho(\cdot)$ via a regularization parameter $\gamma \ge 0$. By considering the robustness of detecting a critical fault $\bar{f}$ of specific interest, the relaxed constraint \eqref{eq:RS_P} offers a new trade-off design scheme as compared to the original design \eqref{eq_DRFD}. 
Unlike the hard constraint \eqref{eq:qp1} on FAR,  a relaxed constraint \eqref{eq:qp2} on FDR is applied to achieve the robustness of detecting $\bar{f}$.  
By introducing multiple relaxed DRCCs into \eqref{eq:RS_P}, the proposed DRFD scheme can be readily extended to tackle multiple critical faults: 
\begin{align*}
\max_{P,\eta_i>0} &~ \rho\left({{P}}\right) +  \sum^{K}_{i=0} {\gamma_i}{\eta_i}   \\
 {\rm s.t. } &~ \sup_{\mathbb{P}_{\xi} \in \mathcal{D}_{\rm W}(\theta, N)}
\mathbb{P}_\xi \left\{\lVert P \xi\rVert^2  > 1 \right\} \leq \alpha  \\
\begin{split} 
&~\sup_{\mathbb{P}_{\xi}  \in \mathcal{M}(\Xi)} \mathbb{P}_\xi \left\{\left\lVert  P \left( \xi + V \bar{f}_i \right)\right \rVert^2 \leq 1\right\}   \\
&~~~\qquad \quad -  {\frac{1}{\eta}} d_{\rm W}(\mathbb{P}_{\xi},\hat{\mathbb{P}}_{N})
  \leq 1 -\beta_i  ,~ \forall  i \in \mathbb{N}_{1:K}\\
\end{split}
\end{align*}
where $\beta_i$ and $\gamma_i$ are tolerance FAR levels and weight parameters corresponding to known faults $\bar{f_i}$. 

\begin{remark}
Another alternative option to replace \eqref{eq: 5} to ensure fault detectability would be to directly impose a DRCC on FDR:
    \begin{equation}
\sup_{\mathbb{P}_{\xi} \in \mathcal{D}_{\rm W}(\theta, N)}  \mathbb{P}_\xi \left\{\left\lVert  P \left( \xi + V \bar{f} \right) \right\rVert^2 > 1\right\} \geq \beta,  \label{eq_DRCC_FDR}
    \end{equation}
which bears a similar spirit to the DRCC \eqref{eq:qp1} on FAR. However, this choice may lead to significant practical difficulties. Note that only $\alpha $ needs to be set in the DRCC on FAR under the routine cases. Unlike the constraint on FAR, it is not clear whether setting a particular target $\beta$ of FDR is suitable under the critical fault profile $\bar{f}$. An improperly chosen $\beta$ can render the optimization problem infeasible. This issue can be well circumvented by using the proposed relaxed DRCC \eqref{eq: 5} by considering all possible distributions. Besides, another merit lies in that the realistic fault situation may more or less deviate from the assumed fault profile, and thus it is useful to have an extra level of robustness ensured by the relaxed DRCC \eqref{eq: 5}.

\end{remark}

\begin{remark}
In DRO, conditional value-at-risk (CVaR) is a common technique to approximate the chance constraint with convex programming reformulations at the cost of some conservatism \citep{hota2019data,xue2022integrated}. However, if we replace the DRCC \eqref{eq:qp2} with a worst-case CVaR constraint, the bilinear matrix inequalities (BMIs) in the reformulation cannot be eliminated.  Consequently, introducing CVaR would only add unnecessary conservatism without offering computational benefits.  For this reason, we focus on deriving an exact reformulation of the original chance constraint \eqref{eq:qp2}. 
\end{remark}

\subsection{Reformulation and Solution Algorithm}
Due to the infinite dimensions of $\mathbb{P}_\xi$ as well as constraints \eqref{eq:qp1} and \eqref{eq:qp2}, the distributionally robust chance-constrained program \eqref{eq:RS_P} is challenging to solve. This subsection seeks to derive a tractable counterpart reformulation alongside an effective solution algorithm. To this end, the following two lemmas are first presented.

\begin{lemma} \label{lemma:rs}
The following optimization problem on an arbitrary non-empty set $\mathcal{S} \in \Xi$, where the decision variable $\mathbb{P}_{\xi}$ is an arbitrary distribution on probability space $\mathcal{M}\left(\Xi\right)$
\begin{equation}\label{lemma_rs}
    \sup_{\mathbb{P}_{\xi} \in \mathcal{M}(\Xi)}\mathbb{P}_\xi\{\xi \in \mathcal{S}\} - {\frac{1}{\eta}} d_{\rm W}(\mathbb{P}_{\xi},\hat{\mathbb{P}}_{N})
\end{equation}
is equivalent to the following problem:
\begin{equation}
\begin{split}
\inf_{v}~ &~ {\frac{1}{N}} \sum^{N}_{i=1} \bar v_i\\
{\rm s.t.} &~ \bar v_i \geq 1 - {\frac{1}{\eta}} \inf_{\xi \in \mathcal{S}}\lVert\xi - \hat{\xi}_i \rVert ,~\forall i \in \mathbb{N}_{1:N}\\
&~ \bar v_i \geq 0,~\forall i \in \mathbb{N}_{1:N}
\end{split} \label{eq:lemma_rs_con}
\end{equation}
\begin{proof}
The proof can be made by adapting that of \cite[Theorem 3.1]{hota2019data} to the problem under study. Define $\kappa={1/ \eta} > 0$. By definition of the Wasserstein distance, \eqref{lemma_rs} can be rewritten as
\begin{equation}\label{eq:theo_ccp}
\begin{aligned}
    &\sup_{\mathbb{P}_{\xi} \in \mathcal{M}(\Xi)}\mathbb{P}_\xi\{\xi \in \mathcal{S}\} - \kappa d_{\rm W}(\mathbb{P}_{\xi},\hat{\mathbb{P}}_{N})\\ 
     =&\sup_{\mathbb{Q} \in \mathcal{Q}}~\mathbb{E}_{\mathbb{Q}}\left\{
     \mathbb{1}_{{\rm cl}(\xi \in \mathcal{S})}
     - \kappa \lVert\xi - \xi' \rVert\right\}
\end{aligned}
\end{equation}
where the ambiguity set $\mathcal{Q}$ is composed of all possible joint distributions of  $\xi$ and $\xi'$ with marginal distributions $\mathbb{P}_\xi$ and $\hat{\mathbb{P}}_
N$:
\begin{equation*}
\left.\mathcal{Q}=\left\{\mathbb{Q}\in\mathcal{M}(\Xi^2) \middle|~
\begin{array}{l}({\xi},{\xi'})\sim\mathbb{Q}, \displaystyle \int_{\Xi} \mathbb{Q}( \xi,{\rm d}{\xi}')=\mathbb{P}_{\xi}(\xi),
\\
\displaystyle \int_{\Xi} \mathbb{Q}( {\rm d}\xi,\hat{\xi}_i)
=\frac{1}{N}, ~  \forall i\in \mathbb{N}_{1:N}\end{array}\right.\right\}.  
\end{equation*}
This allows the variable $\mathbb{P}_{\xi}$ to be implicitly considered by optimizing over $\mathbb{Q}$. Then the semi-infinite program \eqref{eq:theo_ccp} can be reformulated as
\begin{align*}
&\sup_{\mathbb{Q} \in \mathcal{Q}}~\mathbb{E}_{\mathbb{Q}}\left\{
     \mathbb{1}_{{\rm cl}(\xi \in \mathcal{S})}
     - \kappa \lVert\xi - \xi' \rVert\right\}\\
    =& \sup_{\mathbb{Q} \in \mathcal{Q}}
    \int_{ \Xi^2} \left(\mathbb{1}_{{\rm cl}(\xi \in \mathcal{S})}- \kappa \lVert\xi - \xi' \rVert\right) \mathbb{Q}({\rm d}\xi,{\rm d}\xi')  \\
    =&\frac{1}{N} \sum_{i=1}^{N} \sup_{\mathbb{Q}_i \in \mathcal{M}(\Xi)}  \int_{ \Xi} \left[\mathbb{1}_{{\rm cl}(\xi \in \mathcal{S})}- \kappa \lVert\xi - \hat{\xi}_i \rVert\right] \mathbb{Q}_i({\rm d}\xi)\\
    =&\frac{1}{N} \sum_{i=1}^{N} \sup_{\xi \in \Xi}   \left[\mathbb{1}_{{\rm cl}(\xi \in \mathcal{S})}- \kappa\lVert\xi - \hat{\xi}_i \rVert\right] 
\end{align*}
where the second equality is due to the fact that the second marginal distribution of $\mathbb{Q}$ is the empirical distribution $\hat{\mathbb{P}}_N$ and $\{\mathbb{Q}_i\}_{i=1}^N$ are the conditional distributions of $\xi$ given $\xi' = \hat{\xi}_i$. Then one can consider 
each term in the summation over $\mathcal{S}$ and $\Xi \backslash \mathcal{S} $ respectively as
\begin{subequations}
\begin{align}
&\sup_{\xi \in \Xi}   \left[\mathbb{1}_{{\rm cl}(\xi \in \mathcal{S})}- \kappa \lVert\xi - \hat{\xi}_i \rVert\right] \label{eq_16a}\\
=&\max\left\{ 1 - \inf_{\xi \in\mathcal{S}} \kappa \lVert{\xi} - {\hat{\xi}_i} \rVert  ,~ - \inf_{\xi \in 
\Xi \backslash\mathcal{S} } \kappa \lVert{\xi} - {\hat{\xi}_i} \rVert \right\} \label{eq_16b} \\
=&\max\left\{ 1 - \kappa \inf_{\xi \in \mathcal{S}}  \lVert{\xi} - {\hat{\xi}_i} \rVert  ,~ 0 \right\},\label{eq_16c}
\end{align}
\end{subequations}
The equality between \eqref{eq_16b} and \eqref{eq_16c} can be derived by distinguishing the location ${\hat{\xi}_i}$ into two cases:
\begin{itemize}
    \item When $\hat{\xi}_i \in  \mathcal{S}$, the minimal distance from $\hat{\xi}_i$ to $\mathcal{S}$, thus $ 1 - \kappa \inf_{\xi \in  \mathcal{S}}  \lVert{\xi} - {\hat{\xi}_i} \rVert =1$. For given $\kappa \geq 0$, $- \inf_{\xi \in 
\Xi \backslash\mathcal{S} } \kappa \lVert{\xi} - {\hat{\xi}_i} \rVert \leq 0 $, implying that \eqref{eq_16b} $=1$.
    \item When $\hat{\xi}_i \in \Xi \backslash \mathcal{S} $, similarly, $\inf_{\xi \in 
\Xi \backslash\mathcal{S} } \kappa \lVert{\xi} - {\hat{\xi}_i} \rVert = 0 $, so \eqref{eq_16b} equals to $\max\left\{1 - \kappa \inf_{\xi \in \mathcal{S}}  \lVert{\xi} - {\hat{\xi}_i}\rVert   ,~ 0 \right\}$. 
\end{itemize}
Combining two cases, \eqref{eq_16b} simplifies to \eqref{eq_16c}. This yields \eqref{eq:lemma_rs_con} and completes the proof.
\end{proof}
\end{lemma}

Lemma \ref{lemma:rs} reformulates the DRCC in \eqref{eq:RS_P} into semi-infinite constraints. This allows us to further arrive at semi-definite constraints in the following lemma.
\begin{lemma} \label{lemma:qp} 
When $\Xi=\mathbb{R}^{n_{\xi}}$ and  $\mathcal{S} = \{ \xi \mid  \lVert P \left(\xi +V \bar f\right) \rVert^2  \leq 1 \}$ , the optimal value $\inf_{\xi \in \mathcal{S}}  \lVert{\xi} - {\hat{\xi}} \rVert $ is equivalent to solve the following semi-definite program (SDP):
\begin{subequations} 
\begin{align}
    \max_{\pi \geq 0,p,u} &~ u\\
\begin{split} {\rm s.t.}~&~\begin{bmatrix}I&-\hat{\xi}\\-\hat{\xi}^\top&{\hat{\xi}}^\top{\hat{\xi}}-p\end{bmatrix} \\
& +
\pi\begin{bmatrix} \bar P & \bar P V \bar{f}\\\bar{f}^\top V^\top\bar P&\lVert V\bar{f}\rVert_{\bar P}^2-1\end{bmatrix}\succeq 0,~j\in \mathbb{N}_{1:N_j}\end{split} \label{eq_lemma21}\\
&~\begin{bmatrix}p&u\\u&1 \end{bmatrix}\succeq 0 \label{eq_lemma22}
\end{align}\label{eq_lemma2}
\end{subequations}
\end{lemma}

\begin{proof} $\inf_{\xi \in  \mathcal{S}}  \lVert{\xi} - {\hat{\xi}} \rVert $ is equal to that of the following semi-infinite optimization problem:
    \begin{subequations}
       \begin{align}
        \max_{p,u}~&   u\\
         {\rm s.t.}        &~ u^2 \leq p\label{eq_lemma2_1}\\
         \begin{split} 
         &(\xi - \hat{\xi})^\top (\xi - \hat{\xi}) \geq p  , \\
         &\quad   \quad \quad  \forall \xi~ {\rm s.t.} ~ \left(\xi +V\bar{f}\right)^\top \bar P \left(\xi +V\bar{f}\right)   \leq 1 
          \end{split} \label{eq_lemma2_2}
        \end{align}
 \end{subequations}
where the auxiliary scalars $p$ and $u$ are introduced to denote $\inf_{\xi \in   \mathcal{S}} \lVert{\xi} - {\hat{\xi}} \rVert^2$ and $\inf_{\xi \in    \mathcal{S}} \lVert{\xi} - {\hat{\xi}} \rVert$, respectively. This enables the application of the S-procedure \citep{polik2007survey}, which transforms \eqref{eq_lemma2_2} to \eqref{eq_lemma21}. In addition, \eqref{eq_lemma2_1} essentially describes a second-order cone, which can be reformulated as a linear matrix inequality (LMI) \eqref{eq_lemma22}. Combining these two aspects completes the proof. 
\end{proof}
\begin{remark} \label{rmk_boundsup}
Lemma \ref{lemma:qp} applies to the case unbounded support $\Xi=\mathbb{R}^{n_\xi}$. Indeed, our framework can also be extended to the Wasserstein ambiguity set with a bounded support set $ \Xi=\big\{\xi \mid (\xi-c_j)^\top \Omega_j (\xi-c_j) \leq1,~\forall j \in\mathbb{N}_{1:N_j}  \big\}$ with all $\Omega_j \succ0$ in a similar spirit. We assume all samples  $ \hat \xi_i \in \Xi$ and define $\mathcal{S}'=\mathcal{S} \cap \Xi$, by applying S-procedure \citep{polik2007survey}, the inner approximation of $\inf_{\xi \in \mathcal{S}'}  \lVert{\xi} - {\hat{\xi}} \rVert $ is given by:
\begin{equation} 
\begin{aligned}
\begin{split} 
    \max_{\pi \geq 0,\omega \geq 0,p,u} &~ u\\
{\rm s.t.}~&~\begin{bmatrix}I&-\hat{\xi}\\-\hat{\xi}^\top&{\hat{\xi}}^\top{\hat{\xi}}-p\end{bmatrix} +\sum^{N_j}_{j=1} \omega_j \begin{bmatrix} \Omega_j &\Omega_j c_j\\ c_j^\top\Omega_j& c_j^\top\Omega_j c_j-1
\end{bmatrix}\\
& \qquad +
\pi\begin{bmatrix} \bar P & \bar P V \bar{f}\\\bar{f}^\top V^\top \bar P&\lVert V\bar{f}\rVert_{\bar P}^2-1\end{bmatrix}\succeq 0 \\
&~\begin{bmatrix}p&u\\u&1 \end{bmatrix}\succeq 0
\end{split}
\end{aligned} \label{eq_remark_bound}
\end{equation}
By replacing \eqref{eq_lemma2} with \eqref{eq_remark_bound}, one can incorporate support information $\Xi$, thus improving the solution at the expense of increasing computational burden from the additional decision variables $\left\{\omega_j\right\}^{N_j}_{j=1}$.
\end{remark}

Now we are ready to address the DRFD problem \eqref{eq:RS_P}, by establishing the following result.

\begin{theorem} \label{th_1}
For given support set $\Xi= \mathbb{R}^{n_\xi}$, the DRO problem \eqref{eq:RS_P} admits the following exact reformulation:
{\allowdisplaybreaks 
\begin{subequations}
    \begin{align}
    \max_{\begin{subarray}{c}\bar{P},\eta,\lambda,y,v,\\ \tau,\pi,t,u,p,q\end{subarray}} &~ \rho\left({\bar{P}^{1/2}}\right) + \gamma \eta \\
    {\rm s.t. } &~\theta + \frac{1}{N} \sum^{N}_{i=1}y_i \leq \lambda \alpha  \label{eq_20b} \\
&~y_i \geq 0 ,~ i\in\mathbb{N}_{1:N}\\
&~y_i \geq \lambda  - t_i ,~ i\in\mathbb{N}_{1:N}\\
\begin{split} \label{eq_FAR_e}
    &~\begin{bmatrix}I&-\hat{\xi}_i\\-\hat{\xi}_i^\top&{\hat{\xi}}_i^\top\mathbf{\hat{\xi}}_i-q_i\end{bmatrix}\\
    &~-\tau_i\begin{bmatrix} \bar{P} &0\\0&-1\end{bmatrix}\succeq 0,~ i\in\mathbb{N}_{1:N} 
\end{split}\\
&~\begin{bmatrix}q_i&t_i\\t_i&1\end{bmatrix}\succeq 0,~i\in\mathbb{N}_{1:N} \label{eq_20f}  \\
&\frac{1}{N} \sum^{N}_{i=1}v_i \leq (1- \beta) \eta  \label{eq_20g} \\
&~v_i \geq 0 ,~ i\in\mathbb{N}_{1:N}\\
&~v_i \geq \eta -  u_i ,~ i\in\mathbb{N}_{1:N}\\
\begin{split}
&~\pi_i\begin{bmatrix} \bar{P} & \bar{P}^\top V\bar{f}\\
\bar{f}^\top V^\top \bar{P}&\bar{f}^\top V^\top \bar{P} V\bar{f}-1\end{bmatrix}\\
&+\begin{bmatrix}I&-\hat{\xi}_i\\-\hat{\xi}_i^\top&{\hat{\xi}}_i^\top\mathbf{\hat{\xi}}_i-p_i\end{bmatrix}
 \succeq 0,~ i\in\mathbb{N}_{1:N} 
\end{split} \label{eq_20k}\\
&~\begin{bmatrix}p_i&u_i\\u_i&1\end{bmatrix}\succeq 0,~ i\in\mathbb{N}_{1:N} \label{eq_20l}  \\
&~\eta \geq 0,~ \lambda \geq 0,~ \bar{P} \succeq 0\\
&~\tau_i \geq 0,~ \pi_i \geq 0,~ i\in\mathbb{N}_{1:N} 
\end{align} \label{eq:theo1}
\end{subequations}}

\begin{proof}
By applying Lemmas \ref{lemma:rs} and \ref{lemma:qp} and introducing new decision variables by $v_i: = \bar  v_i\eta$ to eliminate as many bilinear terms in \eqref{eq:lemma_rs_con} as possible, we can obtain the equivalent reformulations \eqref{eq_20g}-\eqref{eq_20l} to the DRCC \eqref{eq:qp2} on FDR. Finally, the DRCC \eqref{eq:qp1} on FAR can be reformulated as \eqref{eq_20b}-\eqref{eq_20f}  in virtue of \cite[Theorem 5]{shang2021distributionally}.
\end{proof}
\end{theorem}

In addition, to reduce the conservatism, the bounded support information of $\Xi$ can also be integrated into the DRCC \eqref{eq:qp1} on FAR. A reformulation can be derived similar to that in Remark \ref{rmk_boundsup}.
Despite the exactness of Theorem \ref{th_1}, the reformulation \eqref{eq:theo1} is NP-hard and intractable due to the presence of the BMIs. Thus, a tailored solution algorithm is necessary. Notice that when the Lagrange multipliers $\{\tau_i\}^N_{i=1}$ and $\{\pi_i\}^N_{i=1}$ are fixed, \eqref{eq:mainproblem} becomes an SDP with decision variables $\bar{P},\eta,\lambda$ and $\{y_i,t_i,q_i,v_i,u_i,p_i\}^N_{i=1}$. Importantly, SDPs can be modularly solved by interior-point solvers \citep{toh2018some} such as MOSEK \citep{mosek}, SDPT3 \citep{toh1999sdpt3} and CSDP \citep{borchers1999csdp}. This motivates the development of a heuristic algorithm based on efficient sequential minimization \citep{wang2018sequential,dinh2011combining} to address \eqref{eq:theo1}. As a result, \eqref{eq:theo1} is first expressed as several subproblems:
\begin{equation}\label{eq:mainproblem}
\begin{split} 
\max_{\bar{P} \succeq 0, \eta > 0} &~ \rho\left({\bar{P}^{1/2}}\right) + \gamma \eta \\
{\rm s.t.}&~ \mathcal{J}_1(\bar{P}) \leq 0,~ \mathcal{J}_2(\bar{P},\eta) \leq 0
\end{split}
\end{equation}
where $\mathcal{J}_1(\bar{P})$ is the optimal value of the subproblem:
\begin{equation}
\begin{split}
\min_{\lambda, y, \tau, t, q}&~ \theta + \frac{1}{N} \sum^{N}_{i=1}y_i - \lambda \alpha \\
{\rm s.t.}~~&~y_i \geq 0 ,~y_i \geq \lambda  - t_i,~\tau_i\geq0,~ i\in\mathbb{N}_{1:N}\\
&~\begin{bmatrix}I&-\hat{\xi}_i\\-\hat{\xi}_i^\top&{\hat{\xi}}_i^\top\mathbf{\hat{\xi}}_i-q_i\end{bmatrix} -\tau_i\begin{bmatrix} \bar{P} &0\\0&-1\end{bmatrix}\succeq 0,\\
&~ \qquad \qquad \qquad \qquad \qquad \qquad  i\in\mathbb{N}_{1:N} \\
&~\begin{bmatrix}q_i&t_i\\t_i&1\end{bmatrix}\succeq 0,~ i\in\mathbb{N}_{1:N},~\lambda \geq 0
\end{split}\label{eq_J1}
\end{equation}
and $ \mathcal{J}_2(\bar{P},\eta)$ is the optimal value of the subproblem:
\begin{equation}
\begin{split}
\min_{v, \pi, u, p}&~ \frac{1}{N} \sum^{N}_{i=1}v_i - (1-\beta) \eta \\
{\rm s.t.}~&~v_i \geq 0 ,~v_i \geq \eta -  u_i ,~\pi_i\geq0,~ i\in\mathbb{N}_{1:N}\\
&~\pi_i\begin{bmatrix} \bar{P} & \bar{P} V\bar{f}\\
\bar{f}^\top V^\top \bar{P}&\lVert Vf\rVert ^2_{\bar{P}}-1\end{bmatrix}\\
&~ \qquad \quad ~ +\begin{bmatrix}I&-\hat{\xi}_i\\-\hat{\xi}_i^\top&{\hat{\xi}}_i^\top\mathbf{\hat{\xi}}_i-p_i\end{bmatrix}\succeq 0,~i\in\mathbb{N}_{1:N} \\
&~\begin{bmatrix}p_i&u_i\\u_i&1\end{bmatrix}\succeq 0,~ i\in\mathbb{N}_{1:N} 
\end{split}\label{eq_J2}
\end{equation}

Then, the minimization can be performed, whose idea is to iterate between updating the Lagrange multipliers and maximizing the objective. First, the multipliers $\{\tau_i,\pi_i\}_{i=1}^{N}$ are updated by solving the SDPs \eqref{eq_J1} and \eqref{eq_J2} while holding $\bar{P}$ and $\eta$ constant. Because any pair of feasible $\{ \bar P ,\eta\}$ in the previous iteration remains feasible, updating the multipliers $\{\tau_i,\pi_i\}_{i=1}^{N}$ helps to enlarge the feasible region.
Subsequently, with the updated multipliers $\{\tau_i,\pi_i\}_{i=1}^{N}$ fixed, the SDP \eqref{eq:mainproblem} is solved to find a higher objective function. This iterative process is repeated until convergence.
However, as an approximated solution strategy,  sequential minimization is essentially a coordinate descent method, which may produce a local optimal solution and does not guarantee global optimality. Therefore, it is crucial to initialize the algorithm with a good feasible solution  $\{\bar{P}^{(0)},\eta ^{(0)}\}$. First, it is recommended to choose $\eta ^{(0)}$ sufficiently small to fulfill \eqref{eq:qp2} in the best possible way. Second, to decide $\bar{P} ^{(0)}$, we aim at maximizing its sensitivity to the direction of $\bar{f}$ in the objective subject to the DRCC of FAR. We propose to solve the following problem to initialize \eqref{eq:theo1}:
\begin{equation} 
\begin{aligned}
\max_{\bar{P}^{(0)}} &~ \lVert V \bar{f}  \rVert_ {\bar{P}^{(0)}}^2\\
{\rm s.t.} &~ \sup_{\mathbb{P}_{\xi} \in \mathcal{D}_{\rm W}(\theta, N)}\mathbb{P}_{\xi} \{ \lVert\xi \rVert^2_{\bar{P}^{(0)}}  > 1\} \leq \alpha
\end{aligned} \label{eq_initial}
\end{equation}
Fortunately, \eqref{eq_initial} is easier to solve than \eqref{eq:theo1} and can be effectively solved in virtue of \cite[Algorithm 1]{shang2021distributionally}. After obtaining $\bar{P}^{(0)}$, we inspect whether its FDR under the empirical distribution is higher than $\beta$. If so, the pair $\{\bar{P}^{(0)},\eta^{(0)}=0\}$ must be a feasible solution to \eqref{eq:theo1} and thus the sequential minimization can be performed, as detailed in Algorithm \ref{alg1}. Otherwise, the initial solution $\bar{P}^{(0)}$ to \eqref{eq:theo1} is infeasible and thus a smaller $\beta$ is suggested to improve the feasibility of the problem \eqref{eq:theo1}. In addition, a guideline is provided to inform the selection of all hyperparameters $\{\alpha, \theta,\beta,\gamma \}$ for users.
\begin{itemize}
\item $\alpha$ is the maximal tolerance of FAR that shall be preselected according to user preference. In principle, a high value $\alpha$ may result in ``alarm flood". Conversely, a low value of $\alpha$ can give rise to poor detectability.
\item $\theta$ is the Wasserstein radius, regulating the distributional robustness of the constraint on FAR. A large $\theta$ enhances robustness, but may lead to a more conservative design with potentially lower fault detectability. To calibrate $\theta$, the $K$-fold cross-validation is a viable option, as summarized in Algorithm \ref{alg2}. 
\item $\beta$ is the minimal FDR under the nominal distribution $\hat{\mathbb{P}}_N$ and thus shall be chosen by the user. It should be noted that the FDR is not always kept below $\beta$ when the true distribution $\mathbb{P}^*_\xi
$ deviates from $\hat{\mathbb{P}}_N$. Thus, the value of $\beta$ shall not be too high in order to ensure robustness and avoid infeasibility of \eqref{eq: 5}. 
\item $\gamma$ is a trade-off parameter that can be flexibly adjusted to balance between the overall detectability and the robustness of detecting the particular fault $\bar{f}$. A useful tuning heuristic is to inspect the optimal value of $\eta$ \textit{a posteriori}. If $\eta^*$ is too small, then there is a lack of distributional robustness of detecting $\bar{f}$ and thus one shall turn to a larger $\gamma$ or a smaller $\beta$.
\end{itemize}
\begin{algorithm}
\caption{Solution algorithm for reformulation of new DRFD \eqref{eq:mainproblem}}
\label{alg1}
\begin{algorithmic}[1]
\Require Coefficient matrix $V$, prior fault $\bar{f}$, additive disturbance samples $\{\hat{\xi}_i\}_{i=1}^N$, Wasserstein radius $\theta$, tolerable FAR $\alpha$ and desired FDR $\beta$.  
    \State Solve \eqref{eq_initial} via \cite[Algorithm 1]{shang2021distributionally} and obtain ${\bar{P}^{(0){*}}}$.
 \State Compute ${\rm FDR}_{\rm emp} \leftarrow
 \sum_{i=1}^{N}\mathbb{1}_{{\rm cl}(\lVert \xi_i + V\bar{f}\rVert_ {\bar{P}^{*}}>1)}/{N}$.\label{alg1_2}
\If{$\rm FDR_{\rm emp} < \beta$} \label{alg1_3}
    \State  \textbf{Return} infeasibility. \label{alg1_4}
    \Else
    \State Initialize $\bar{P}^{(0)} \leftarrow  {\bar{P}^{(0)}{*}}$, $\eta^{(0)}\leftarrow 0$ and $n_{\rm ite}\leftarrow 0$.
\While{Not convergent}
    \State Solve \eqref{eq_J1} with $\bar{P} \leftarrow \bar{P}^{(n_{\rm ite})}$ and obtain $\tau^{(n_{\rm ite})}$.
    \State Solve \eqref{eq_J2} with $\bar{P}  \leftarrow \bar{P}^{(n_{\rm ite})}$ and $\eta  \leftarrow \eta^{(n_{\rm ite})}$, and obtain $\pi^{(n_{\rm ite})}$.
    \State Solve \eqref{eq:mainproblem} with $\tau \leftarrow \tau^{({\rm ite})}$, $\pi \leftarrow \pi^{({\rm ite})}$ and obtain the solution $\bar{{P}}^{(n_{\rm ite}+1)}$ and  $\eta^{(n_{\rm ite}+1)}$ .
    \State $n_{\rm ite} \leftarrow n_{\rm ite} + 1$.
\EndWhile
\State \textbf{Return} $\bar{P}$ and $\eta$.
\EndIf 
\end{algorithmic}
\end{algorithm}

\begin{algorithm}
\caption{Cross-validation for selecting the radius $\theta$}
\label{alg2}
\begin{algorithmic}[1]
\Require Number of folds $K$, samples $\{\hat{\xi}_i\}_{i=1}^N$, grid $\Theta$ of Wasserstein radius, and tolerable FAR $\alpha$.
\State Partition the samples  $\{\hat{\xi}_i\}_{i=1}^N$ into $K$ equally sized folds.
\For{Each candidate value $\theta \in \Theta$}
\For{$k=1:K$}
\State Use the $k$th fold of samples as a test dataset, merge the remaining ones to construct $\mathcal{D}_{\rm W}({\theta};(K-1)N/K)$ and then solve \eqref{eq_DRFD}.
\State Compute ${\rm FAR}_k$ on the $k$th fold.
\EndFor
\State Calculate ${\rm FAR}({\theta}) \leftarrow \sum_{k=1}^{K}{\rm FAR}_k/K$.
\EndFor
\State Find $\theta^* \leftarrow  \displaystyle \arg \max_{{{\theta}}\in \Theta}~ {{\theta}} ~~~
{\rm s.t.}~{\rm FAR}({\theta}) \leq \alpha
$.
\State \textbf{Return} $\theta^*$.
\end{algorithmic}
\end{algorithm}

\section{Case Study on A Simulated Three-tank System}\label{simu1}
First, the efficacy of the proposed DRFD trade-off design scheme is verified on a simulated three-tank system, which has been widely used as a benchmark in the field of FD \citep{ding2008model,shang2021distributionally,shang2022generalized,xue2022integrated,ding2014data,ding2014dsubspace,ding2019application}. As shown in Fig. \ref{3tank}, three water tanks are interconnected with pipes, and two pumps feed water to the tanks. The system input $u(k)=\left[
    u_1~u_2\right]^\top \in \mathbb{R}^2$ is the flow rates of pumps $1$ and $2$, the system states $x(k)=\left[  x_1~x_2~x_3\right]^\top \in \mathbb{R}^3$ is the levels of three tanks, and the outputs $y(k)=\left[
    x_1~x_3\right]^\top \in \mathbb{R}^2$ is the measurement of the levels of Tanks $1$ and $3$. 
By linearizing around the operating point $x_1 = 15$cm, $x_2 = 10$cm, $x_3 = 13$cm  with a sampling interval $\Delta x = 5s$, state-space matrices in  \eqref{eq_statepace} are derived as follows \citep{shang2021distributionally,ding2008model}:
\begin{align*}
&A=\begin{bmatrix}0.8945&0.0048&0.1005\\0.0048&0.8500&0.0801\\0.1005&0.0801&0.8164\end{bmatrix},\\
&B=\begin{bmatrix}0.0317&0.0001\\0.0001&0.0309\\0.0018&0.0014\end{bmatrix},\\
&B_d=I_3,~B_f=\begin{bmatrix}I_3&0_{3\times2}\end{bmatrix},~C=\begin{bmatrix}1&0&0\\0&0&1\end{bmatrix},\\
&D=0,~D_d=0,~D_f=\begin{bmatrix}0_{2\times3}&I_2\end{bmatrix}.
\end{align*}
In our simulation, the disturbance~$d$ is generated from a Laplace distribution with mean $\mu_d=0$ and covariance $\Sigma_d = 0.4I$. The motivation is to create a more challenging scenario that reflects real-world conditions. Unlike the Gaussian distribution, the Laplace distribution is heavy-tailed and can suitably describe the presence of large-magnitude outliers often encountered in practical systems \citep{yu2023robust,yu2025robust}.

\begin{figure}[htbp]
    \centering
    \includegraphics[width=0.95\linewidth]{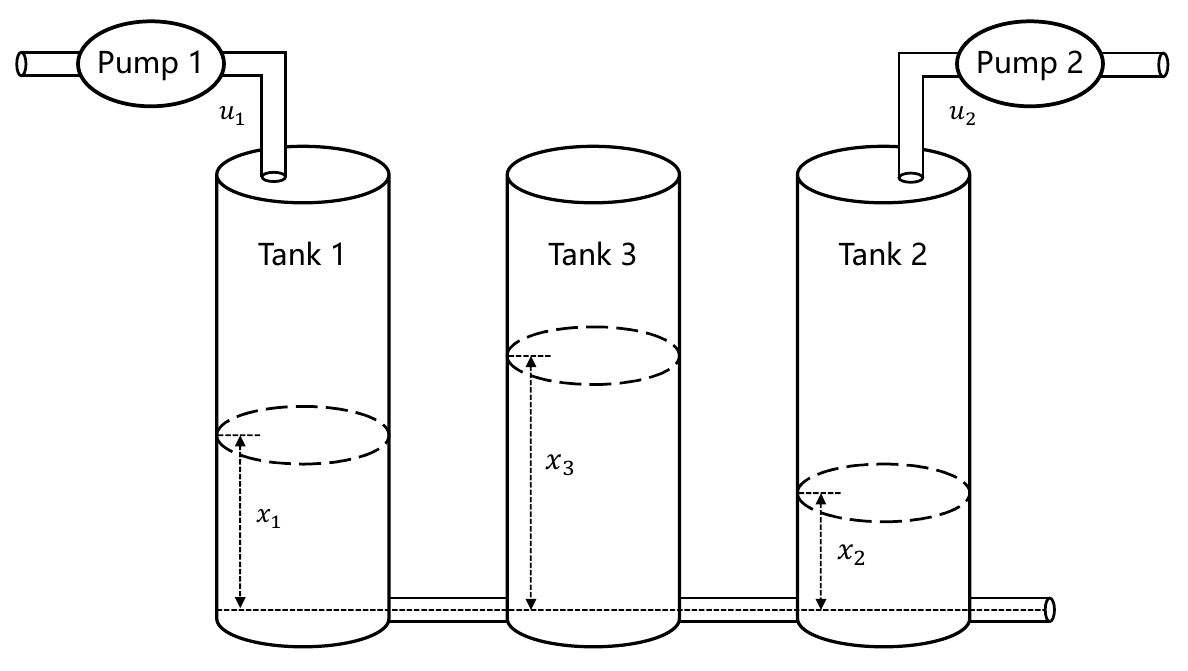}
    \caption{Schematic of the three-tank system.}
    \label{3tank}
\end{figure}

We seek to design the matrix $P$ in \eqref{eq:resi_s}, where the parity space method \citep{chow1984analytical} with order $s = 4$ is used to generate residuals $\{\hat{\xi}_i\}_{i=1}^N$ as samples of uncertainty $\xi$ and attain the matrix $V$. The true distribution $\mathbb{P}^*_\xi$ of uncertainty is assumed to be unknown, but $N=200$ data samples are collected under routine fault-free conditions to capture the inexactness of the true distribution $\mathbb{P}^*_\xi$ using the Wasserstein ambiguity set. We set $\alpha=1\%$ and calibrate the radius $\theta$ of the ambiguity set $\mathcal{D}\left(\theta; N \right)$ using five-fold cross-validation. FARs under different radii within $[10^{-3}, 10^{-2}]$ are shown in Fig. \ref{radius}. Based on this, $\theta_1=0.008$ and $\theta_2=0.009$ are set respectively for $\rho_1(\cdot)$ in \eqref{eq_rho1} and  $\rho_2(\cdot)$ in \eqref{eq_rho2} to ensure the average FAR below the tolerance level $\alpha$.
In the design phase, it is assumed that a single critical fault is given by $\bar{f} = \left[-2~0~0~0~0\right]$, which represents a constant leakage in Tank $1$. Using $\rho_i(\cdot)$ ($i=1,2$), the existing Wasserstein-based DRFD design based on \eqref{eq_DRFD} \citep{shang2021distributionally} and our proposed new DRFD design are denoted by DR$i$ and NDR$i$, respectively. In addition, we also introduce a benchmark, denoted by Init$i$, which uses only the initial solution obtained from \eqref{eq_initial}.
All induced SDPs are solved in MATLAB R2024a with a CPU of Intel Core i7-10700 using MOSEK \citep{mosek} via the YALMIP interface \citep{lofberg2004yalmip}. Our source code is available at https://github.com/Yulin-F/DRFD-prior.

\begin{figure}[htbp]
    \centering
    \includegraphics[width=0.95\linewidth]{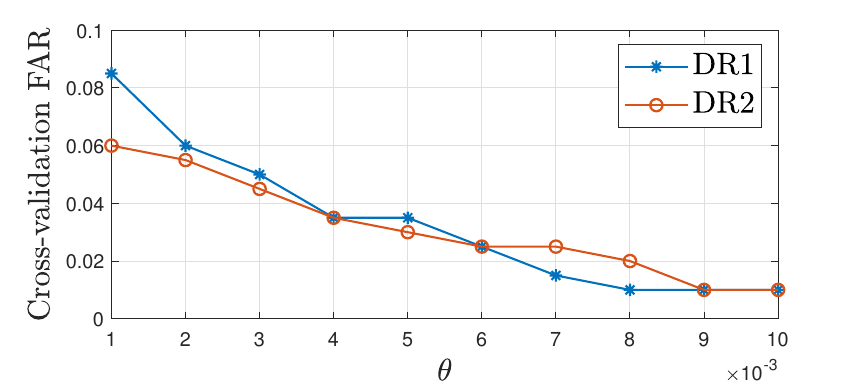}
    \caption{Cross-validation results under different Wasserstein radii and overall detectability}
    \label{radius}
\end{figure}

Setting different values of $\{\beta, \gamma\}$, trade-off curves between the overall detectability $\rho(\cdot)$ and the distributional robustness index $\eta$ of detecting $\bar{f}$ are derived and displayed in Figs. \ref{fig_rho1} and \ref{fig_rho2}. It can be seen that for a given $\beta$, the proposed approach enables to flexibly balance between two design objectives by varying $\gamma > 0$. With $\gamma$ increasing, the distributional robustness of detecting $\bar{f}$ is enhanced at the cost of lower values of $\rho(\cdot)$. Besides, when $\beta = 0.99$, an extremely small $\eta$ is derived, indicating a lack of distributional robustness of detecting $\bar{f}$. This indicates that $\beta$ shall not be chosen too large.
\begin{figure}[htbp]
    \centering    \includegraphics[width=0.9\linewidth]{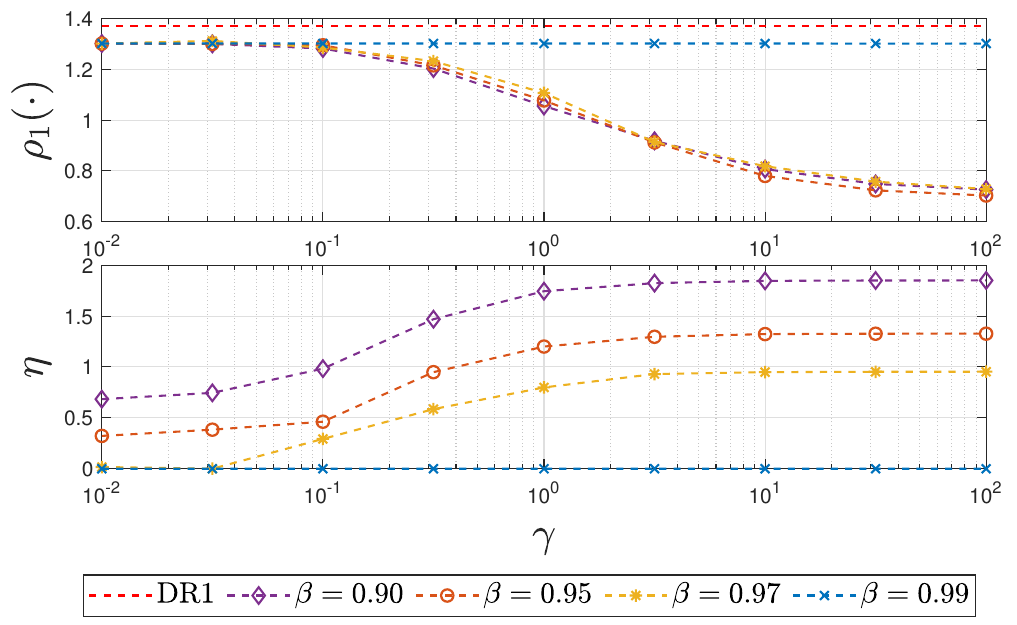}
    \caption{Trade-off between $\rho(\cdot)$ and $\eta$ in NDR$1$}    \label{fig_rho1}
\end{figure}

\begin{figure}[htbp]
    \centering    \includegraphics[width=0.9\linewidth]{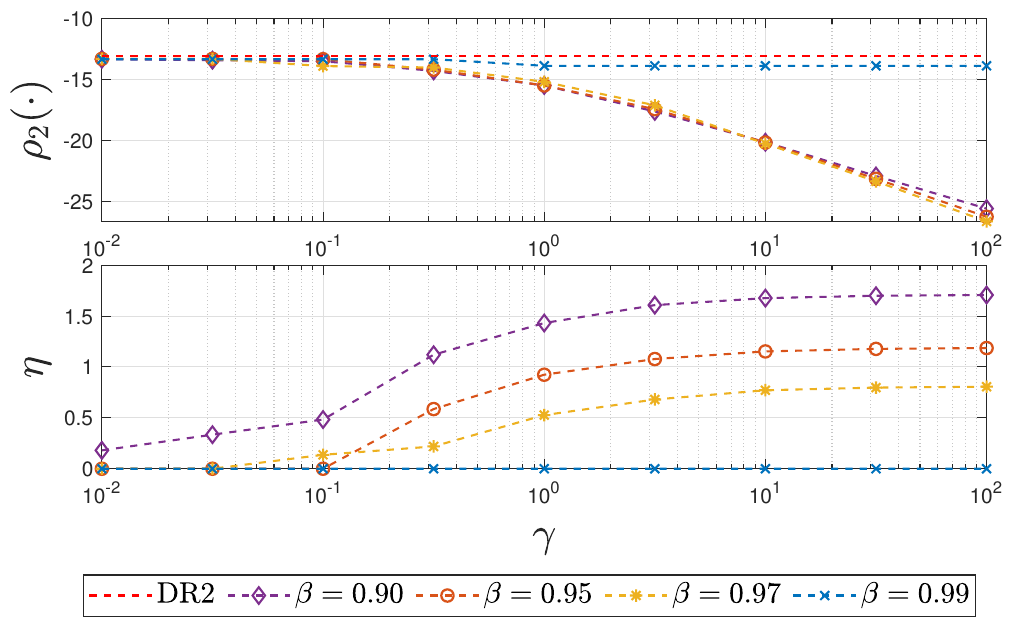}
    \caption{Trade-off between $\rho(\cdot)$ and $\eta$ in NDR$2$}    \label{fig_rho2}
\end{figure}

Monte Carlo simulations are carried out to evaluate the performance of the DRFD designs under fault-free and faulty conditions. We consider the detection of two faults, i.e. $f_1=\bar{f}$ and $f_2=\left[0~5~10~0~0\right]$, the latter of which represents an unknown constant leakage in Tanks 2 and 3. In each Monte Carlo run, an input-output trajectory with $800$ samples is generated and the fault is injected from the $401$st sample. 
Five different choices of $\beta$ and $\gamma$ are used, giving rise to various residual generators, namely NDR$i$-A ($\beta = 95\%$, $\gamma=0.01$), NDR$i$-B ($\beta = 95\%$, $\gamma=0.1$),  NDR$i$-C ($\beta = 95\%$, $\gamma=1$),  NDR$i$-D ($\beta = 97\%$, $\gamma=1$) and NDR$i$-E ($\beta = 99\%$, $\gamma=1$). The average FARs and FDRs computed from $250$ Monte Carlo simulations are given in Table \ref{tab_FAR}, where FDR$i$ indicates the FDR of the fault profile $f_i$, $i=1,2$. In addition, Fig. \ref{fig_residual} shows the profile of the evaluation function $J(r)$ in a single Monte Carlo run. It can be seen that the FARs of all DRFD designs are below the maximal tolerance $\alpha$. For generic DRFD designs DR$1$ and DR$2$, their FDR$1$ are not sufficiently high. The initial solution $\bar{P}^{(0)}$ exhibits a high FDR$1$, but almost completely fails to detect the unknown faults $f_2$. By contrast, all values of FDR$1$ induced by NDR$i$ are higher than $95\%$ while ensuring that FDR$2$ remains at a high level, thereby demonstrating the necessity of Algorithm \ref{alg1}. This also highlights that a remarkably improved detectability against the known fault $f_1$ is enabled by encoding critical fault information. It is also worth adding that considering information of a critical fault does not always lead to higher detectability against all possible faults. For example, the FDRs against $f_2$ of our proposed method are lower than those of generic DRFD, as shown in Table \ref{tab_FAR}. Despite this, a more balanced detection of $f_1$ and $f_2$ is achieved by the proposed DRFD design than the generic scheme in \eqref{eq_DRFD}. 

Next, we discuss the effect of $\gamma$ on the overall detectability and the detectability against the critical fault $\bar{f}$, by comparing the performance of NDR$i$-A, NDR$i$-B, and NDR$i$-C in Table \ref{tab_FAR}. With $\beta = 95\%$ fixed, increasing $\gamma$ results in larger values of $\eta$ and always helps to empirically improve the detectability of $f_1$ in this case. With $\gamma$ decreasing and the overall detectability $\rho$ increasing, the detectability against a particular fault $f_2$ increases, as indicated by ${\rm FDR2_{\rm NDR2-A}}>{\rm FDR2_{\rm NDR2-B}}>{\rm FDR2_{\rm NDR2-C}}$. However, we cannot expect FDR$2$ to behave strictly oppositely with $\gamma$. This is because $\rho$ is a geometric quantification of the overall detectability against all possible unknown faults, so a larger value of $\rho$ does not necessarily indicate higher detectability for $f_2$. As an example, ${\rm FDR2_{\rm NDR1-B}}$ is slightly higher than ${\rm FDR2_{\rm NDR1-A}}$.

The generic DRFD design has lower FDRs than NDR$i$ in all three faulty cases. First, the injection of prior fault information in NDR$i$ leads to an effective improvement of detectability for the known fault $f_1$. Meanwhile, the FDR$2$ and FDR$3$ also increase. A possible explanation is that due to the specific dynamics of the battery cell system, incorporating the prior information of $f_1$ can improve the detectability of a wide range of authentic faults including $f_2$ and $f_3$.
In addition, NDR$1$ excels in detecting $f_1$ and whereas NDR$2$ is superior for $f_2$. This can be explained by the fact that $f_2$ lies closer to the kernel of $V$, so the balanced detectability promoted by $\rho_2(\cdot)$ is more effective, which aligns with the discussion in Remark \ref{remark_rho}.

Then, we discuss the role of $\beta $, which is a target of the FDR for the known fault $f_1$ under the empirical distribution $\hat{\mathbb{P}}_N$. When $\beta$ is chosen as $95\%$ or $97\% $, a superior FDR$1$ is obtained in NDR$i$-C and NDR$i$-D, $i=1,2$. However, a reduction of FDR$1$ can be observed using $\beta=99\%$ in NDR$i$-E. This is because an overly high $\beta$ may compromise the distributional robustness $\eta$ against detecting $f_1$, which leads to a lower detectability under the unknown realistic distribution.

Finally, the convergence performance of Algorithm \ref{alg1} is depicted in Fig. \ref{conv}, which shows that the convergence condition is met only within $5$ iterations. Furthermore, the relationship between solution time and sample $N$ with fixed $\theta$ is presented in Fig. \ref{fig_time}.  The solution time appears to grow exponentially with the sample number $N$. For applications with large datasets, the significant computational burden motivates the use of scenario reduction to improve solution efficiency by selecting a representative subset of samples \citep{rujeerapaiboon2022scenario,NEURIPS2022_a8e0abdd}. In addition, NDR$1$ generally is computationally more efficient than NDR$2$, because $\rho_1(\cdot)$ is linear in $\bar P$, resulting in a more tractable problem. 

\begin{table*}[htpb] 
\centering
\caption{Average FAR and FDR Performance on the Simulated Three-Tank Process}\label{tab_FAR}
\begin{tabular}{lcccccccccccc}
\toprule
&Init$1$& DR$1$ & NDR$1$-A & NDR$1$-B& NDR$1$-C& NDR$1$-D& NDR$1$-E \\
\midrule
FAR (\%)&0.12 &0.27&0.36&0.38&0.37&0.38&0.36\\
FDR$1$ (\%)&98.89&3.52&96.34&97.05&99.10&99.06 &96.22\\
FDR$2$ (\%)& 12.01&96.17&95.92&96.69&75.87&74.95&95.91\\
\midrule
&Init$2$& DR$2$ & NDR$2$-A & NDR$2$-B& NDR$2$-C& NDR$2$-D& NDR$2$-E \\
\midrule
FAR (\%)& 0.09 &0.19&0.24&0.24&0.26&0.27&0.24\\
FDR$1$ (\%)& 98.60&65.07&96.91&97.01&98.54&98.43&96.90\\
FDR$2$ (\%)& 10.16&97.29&96.15&95.66&85.26&89.03&96.17\\
\bottomrule
\end{tabular}
\end{table*}

\begin{figure*}[htbp]   
    \centering    
        \subfloat[FD Designs using overall detectability metric $\rho_1(\cdot)$]{
        \includegraphics[width=0.9\linewidth]{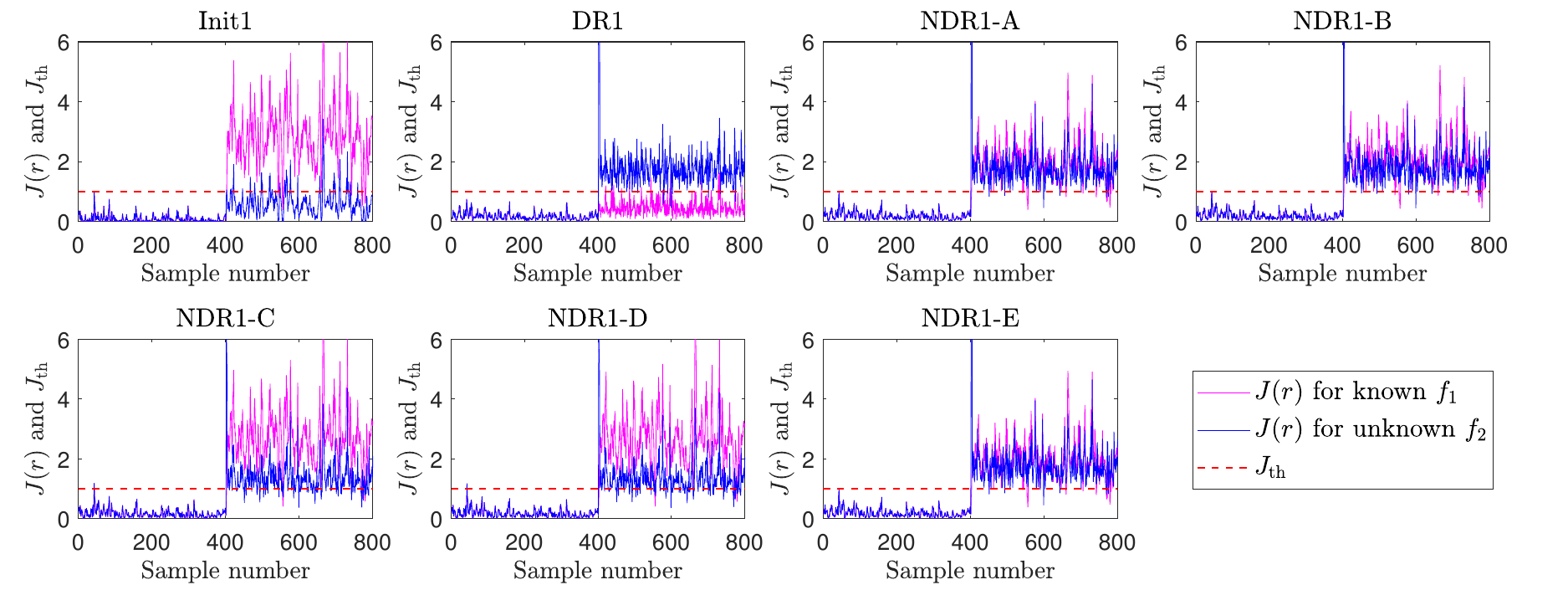} 
        \label{fig_res1}
    }\hspace{0.5cm}
              \subfloat[FD Designs using overall detectability metric $\rho_2(\cdot)$]{
        \includegraphics[width=0.9\linewidth]{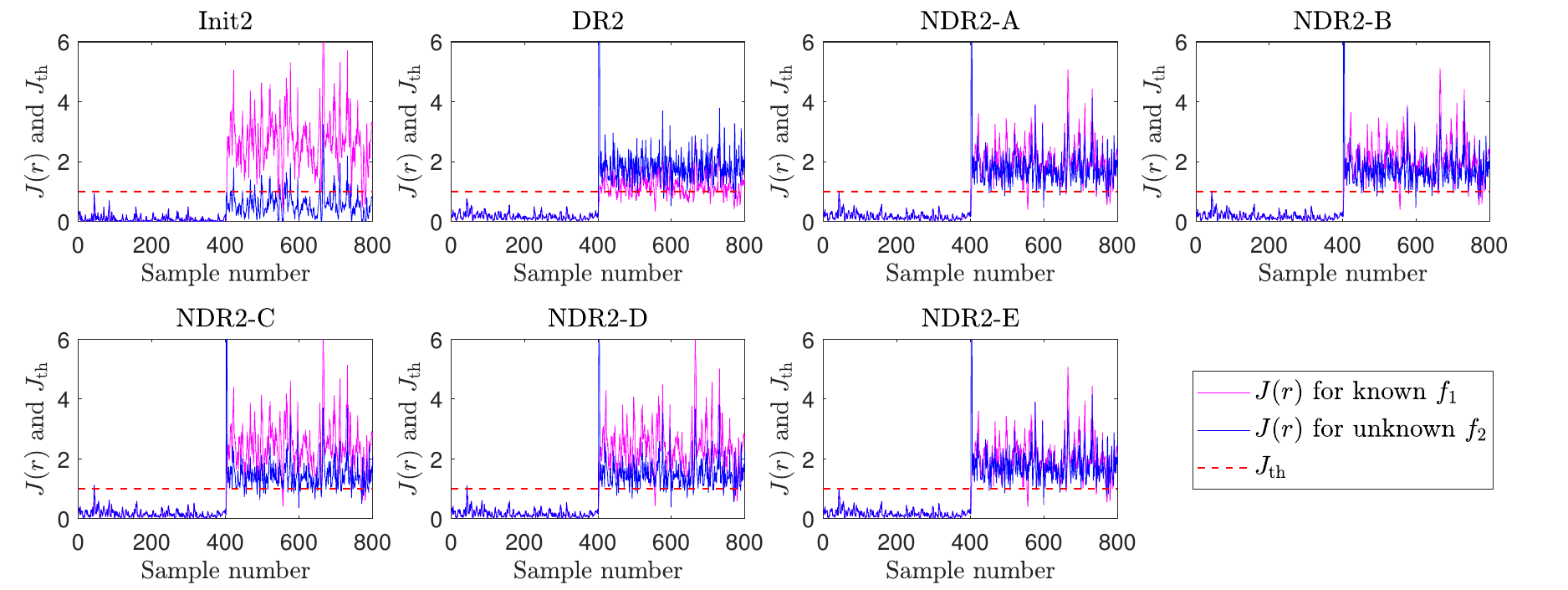} 
        \label{fig_res2}
    }
    \caption{FD results on the simulated three-tank system}
    \label{fig_residual}
\end{figure*}

\begin{figure}[htbp]
    \centering
    \includegraphics[width=0.95\linewidth]{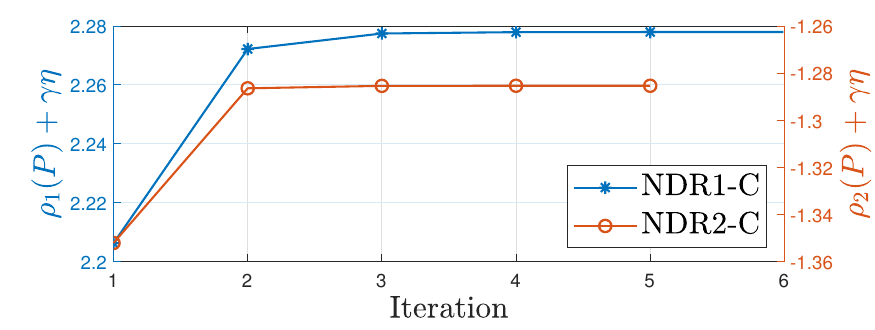}
        \caption{The objective function iteration of Algorithm \ref{alg1}}
    \label{conv}
\end{figure}

\begin{figure}[htbp]
    \centering
    \includegraphics[width=0.95\linewidth]{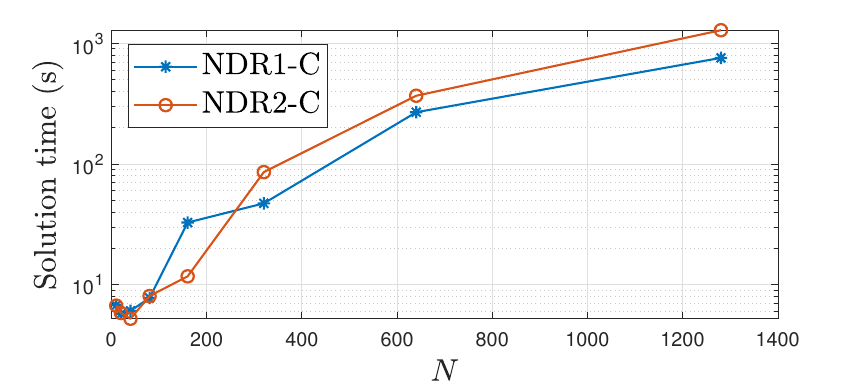}
    \caption{Computation time of Algorithm \ref{alg1} for varying sample sizes $N$}
    \label{fig_time}
\end{figure}
\section{Case Study on A Realistic Lithium-ion Battery Cell}\label{simu2}
Due to the substantial energy density and high efficiency, lithium-ion batteries have found widespread applications in power grids, electric vehicles and portable electronics \citep{jin2023combined,liu2016structural}. Thus, there is a critical need for advanced FD techniques to ensure the safe operation of battery cells \citep{wang2024voltage,xiong2019sensor}. In this section, we perform the case study using real-world data collected from the charging of a lithium-ion cell. Based on Thevenin's theorem, the equivalent battery circuit model \citep{xiong2019sensor,wang2024voltage} consists of an ohmic resistance $R_{\rm o}$,  a resistor-capacitor pair$\{R_{\rm par}, C_{\rm par}\}$ and open-circuit voltage $U_{\rm oc}$, which are shown in Fig. \ref{battery}. According to Kirchhoff’s law, its dynamics can be written as:
\begin{equation*}\begin{cases}
\dot{U}_{\rm par} = -\frac{1}{C_{\rm par} R_{\rm par}} U_{\rm par} + \frac{1}{C_{\rm par}} I_{\rm L}\\
\dot{S}=-\frac{1}{ Q}I_{\rm L}\\
U_{\rm t} = U_{\rm oc} - U_{\rm par} - I_{\rm L} R_{\rm o} &\end{cases} 
\end{equation*}
where $U_{\rm par}$, $I_{\rm L}$, $Q$, $S$ and $U_{\rm t}$ are the voltage of $R_{\rm par}$, the battery current across the battery, the maximum charge capacity, the state-of-charge and the terminal voltage. $U_{\rm oc}$ is commonly modeled as a polynomial function related to the  $S$ \citep{jin2023combined,xiong2019sensor}. Thus, the charging process can be represented as a nonlinear system with state $\left [ U_{\rm par} ~S \right]^\top$, input $I_{\rm L}$ and output $U_{\rm t}$.
\begin{figure}[htbp]
    \centering
    \includegraphics[width=0.95\linewidth]{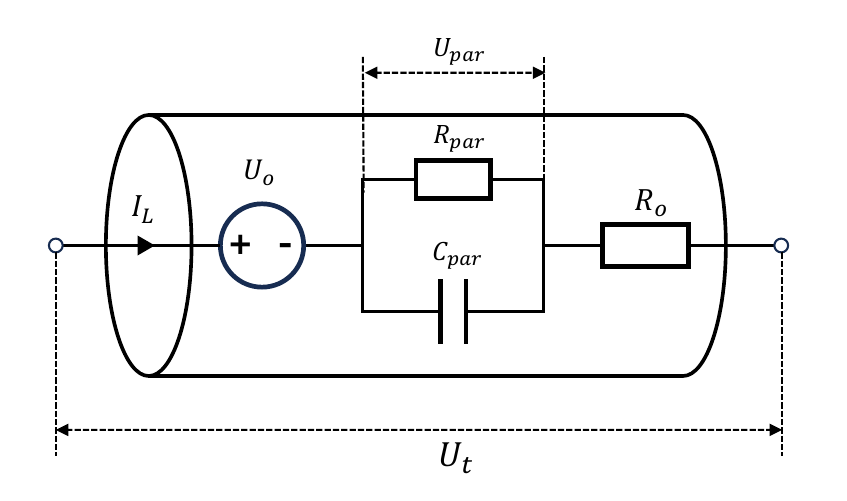}
\caption{ Equivalent circuit model of a single lithium-ion cell}
    \label{battery}
\end{figure}

As illustrated in Fig. \ref{fig_platform}, the real-world experimental platform consists of a temperature tester, a battery tester, a thermostat, and a computer, where all data, including voltage, current, and temperature, are monitored and sampled at a frequency of $1$Hz. More details about the experimental platform setup can be found in \cite{jin2023combined}. The input and output of the system are the current across the battery and the terminal voltage. 
\begin{figure}[htbp]
    \centering
    \includegraphics[width=0.8\linewidth]{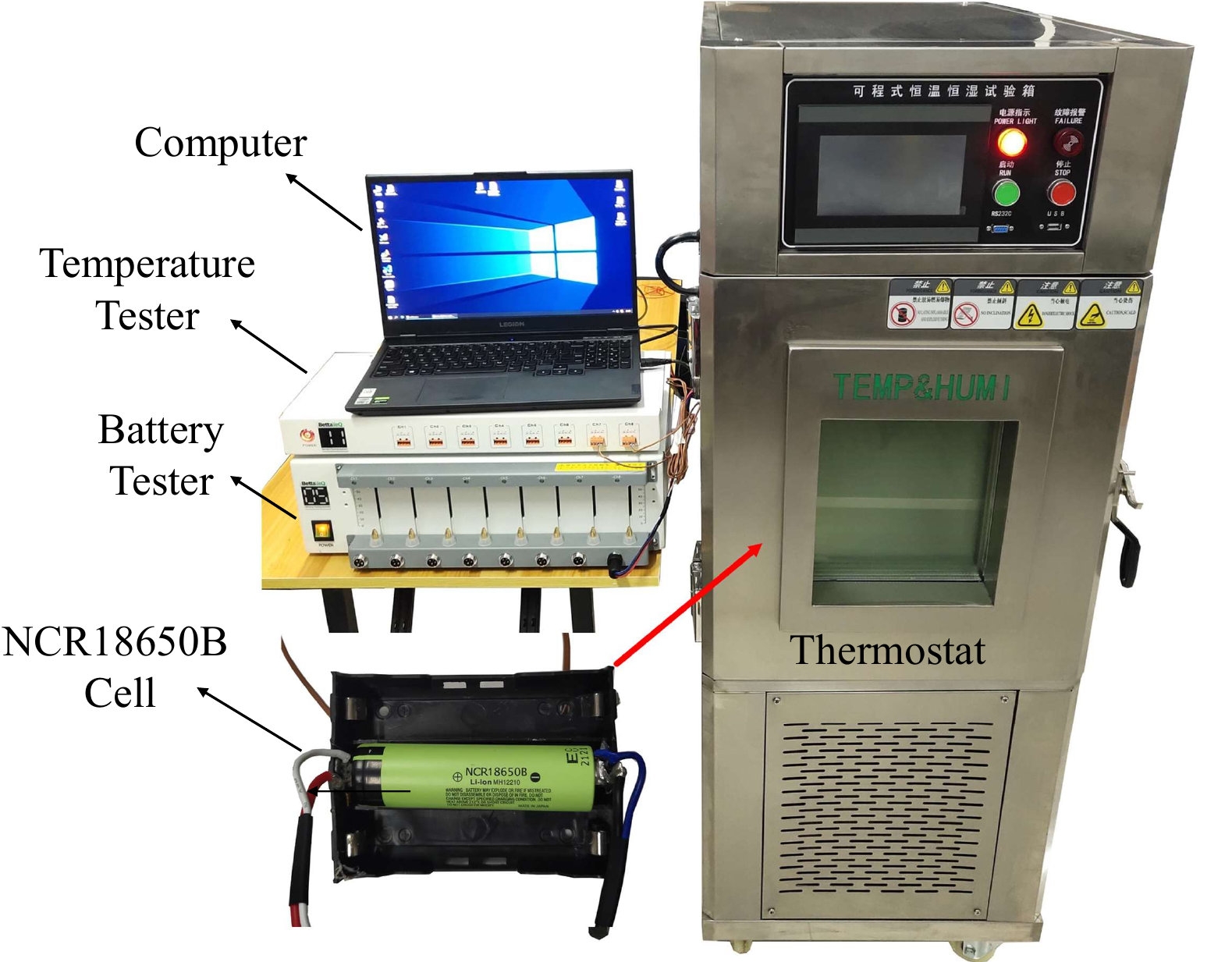}
    \caption{Battery cell experimental platform}
    \label{fig_platform}
\end{figure}

A data-driven SKR-based residual generator \cite[Algorithm 3.3]{ding2014dsubspace} is implemented based on an input-output data matrix with depth $s=5$ and system order $n_x=2$. We use $N = 200$ samples of residuals to construct the Wasserstein ambiguity set. A $180~{\rm mA}$ actuator bias on $I_{\rm L}$ is considered as a single critical fault $f_1 = \bar{f}$, which is known in the design phase. We set $\alpha=2\%$, $\beta = 90\%$, and use different values of $\gamma$, thereby yielding three DRFD designs, i.e., NDR$i$-A ($\gamma = 10$), NDR$i$-B ($\gamma = 100$) and NDR$i$-C ($\gamma = 500$). To evaluate the fault detection performance, we also take into account another two unknown faults:
\begin{itemize}
    \item Fault $f_2$: A $250~{\rm mV}$ sensor bias on the terminal voltage $U_t$.
    \item Fault $f_3$: A hybrid fault comprising a $-100~{\rm mA}$ actuator bias on $I_{\rm L}$ and a $120~{\rm mV}$ sensor bias on $U_t$.
\end{itemize}

\noindent An input-output data trajectory is collected from a real-world experiment, and the fault is injected from the $151$st sample into the system till the end. 

The computed FARs, FDRs and robustness indices are presented in Table \ref{tab_FAR2}, and the profile of evaluation function $J(r)$ is shown in Fig. \ref{fig_residual_real}.
It is clear that the FARs are safely controlled below the tolerance $\alpha$ in each DRFD design. The generic DRFD design has lower FDRs than NDR$i$ in all three faulty cases. First, the injection of prior fault information in NDR$i$ leads to an effective and significant improvement of detectability for the known fault $f_1$. 
Interestingly, unlike the three-tank system where performance on $f_2$ degraded, this case study illustrates that prioritizing $f_1$ can sometimes coincidentally improve the detectability of another faults $f_2$ and $f_3$. A possible explanation is that enhancing the detectability for known $f_1$ does not guarantee a specific outcome for an unknown fault. That is, this impact may be either beneficial or detrimental depending on the nature of another fault and the dynamic of the system. In addition, the increased $\gamma$ typically results in better detection of $f_1$ but does not always improve the detectability of $f_2$ and $f_3$, mainly because the distributional robustness $\eta$ of detecting $f_1$ is improved at the sacrifice of the overall detectability. 

\begin{figure*}[htbp]   
    \centering    
        \subfloat[FD designs using overall detectability metric $\rho_1(\cdot)$]{
        \includegraphics[width=0.95\linewidth]{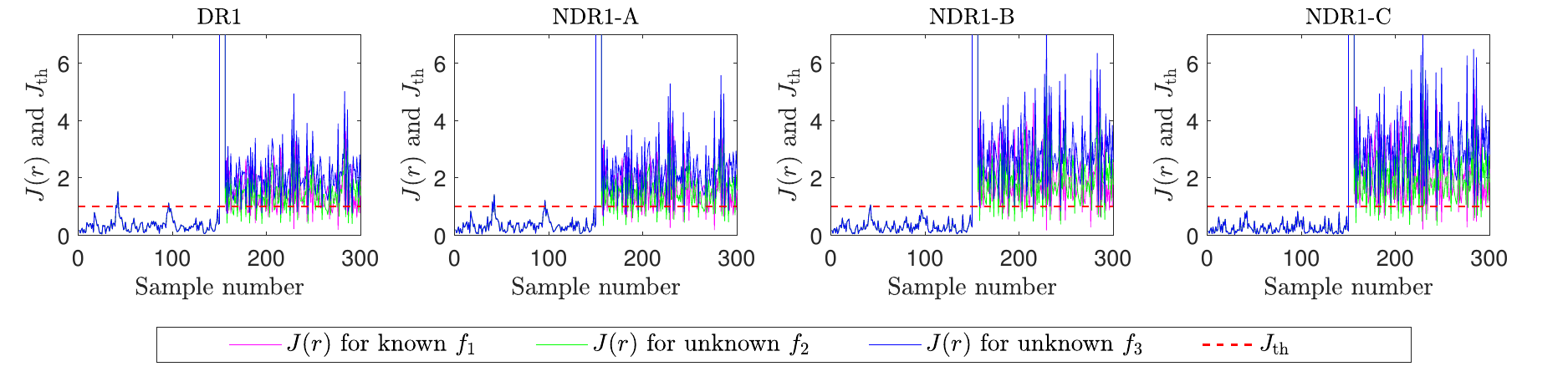} 
        \label{fig_residual_real_1}
    }\hspace{0.5cm}
              \subfloat[FD designs using overall detectability metric $\rho_2(\cdot)$]{
        \includegraphics[width=0.95\linewidth]{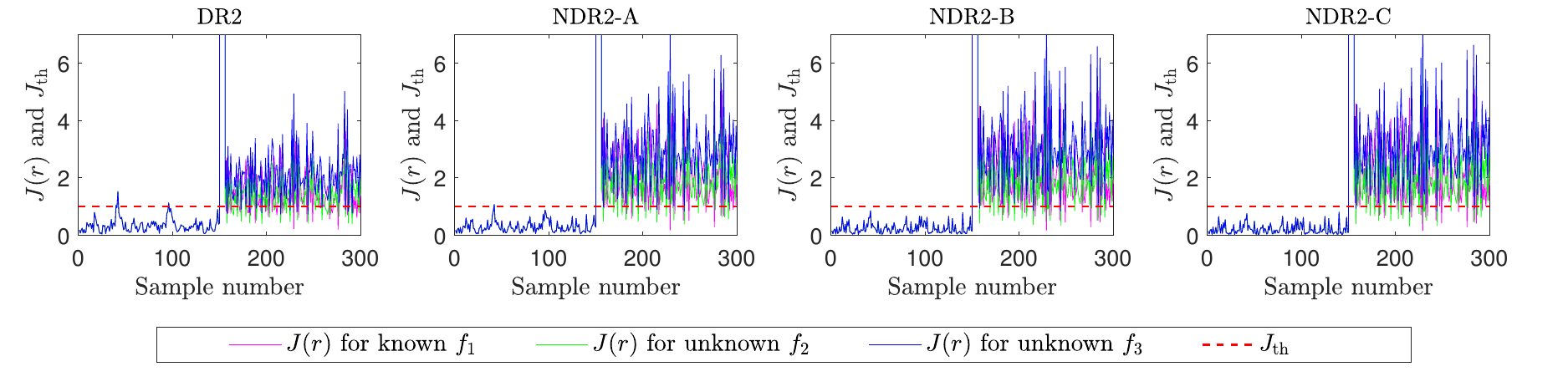} 
        \label{fig_residual_real_2}
    }
    \caption{FD results on the real battery cell charging}
    \label{fig_residual_real}
\end{figure*}

\begin{table*}[htpb] 
\centering
\caption{
FAR, FDR and Distributional Robustness of DRFD Designs on Battery Cells}\label{tab_FAR2}
\begin{tabular}{lcccccccc}
\toprule
        & DR$1$ & NDR$1$-A & NDR$1$-B& NDR$1$-C& DR$2$ & NDR$2$-A & NDR$2$-B& NDR$2$-C\\
\midrule
FAR (\%)& 2.00&2.00&0.67&0.00&2.00&0.67&0.00&0.00\\
FDR$1$ (\%)&77.33&78.67&87.33&88.00&77.33&87.33&88.00&88.00\\
FDR$2$ (\%)&75.33&75.33&86.67&86.67&75.33&86.67&85.33&84.67 \\
FDR$3$ (\%)&91.33&93.33&98.67&97.33&91.33&98.00&96.67&96.67\\
$\eta$&/&
0.0499& 0.1558&  0.1728& /&   0.1578  &  0.1740  &  0.1761\\
\bottomrule
\end{tabular}
\end{table*}

\section{Conclusion} \label{concl}
In this work, a new DRFD design scheme was proposed that incorporates prior information of critical faults and safely keeps FAR under control without exactly knowing the probabilistic distribution of uncertainty. Built upon a new relaxed distributionally robust chance constraint, the proposed approach makes it feasible to balance between a robust detection of critical faults and maintaining the detectability of unknown faults. Then, an exact reformulation of the proposed new DRFD problem was derived as a semi-definite program with BMIs. We developed a tailored heuristic algorithm, including an initialization strategy and a sequential minimization procedure, which helps to find a high-quality solution of the problem by solving SDPs. Finally, comprehensive case studies on a simulated three-tank system and a realistic lithium-ion battery cell showed that our proposed heuristic solution algorithm can efficiently find a high-quality design, which achieves a more balanced performance than generic DRFD designs. It not only maintains good sensitivity against unknown faults but also ensures a robust detection of known critical faults. In future work, it is worth investigating how to relax the dependence on the exact fault information and further address time-varying disturbances by injecting more probabilistic information.




\section*{Declaration of Competing Interest}
The authors declare that they have no known competing financial interests or personal relationships that could have appeared to influence the work reported in this paper.

\section*{Acknowledgments}
This work was supported by the National Natural Science Foundation of China (Grant No. 62373211), and the Open Research Project of the State Key Laboratory of Industrial Control Technology, China (Grant No. ICT2025B10).





\bibliographystyle{cas-model2-names}

\bibliography{ref}

\end{document}